\documentclass[final,3p]{elsarticle}
 \usepackage{graphics}
 \usepackage{graphicx}
 \usepackage{epsfig}
\usepackage{amssymb}
 \usepackage{amsthm}
 \usepackage{lineno}
 \usepackage{amsmath}
   \numberwithin{equation}{section}
\usepackage{mathrsfs}
\newtheorem{thm}{Theorem}[section]

\newtheorem{lem}[thm]{Lemma}

\newtheorem{defn}[thm]{Definition}

\biboptions{sort&compress,square}
\allowdisplaybreaks

\begin{document}
\begin{frontmatter}
\author[rvt1,rvt2]{Jin Hong}
\ead{jhong@nenu.edu.cn}
\author[rvt2]{Yong Wang\corref{cor2}}
\ead{wangy581@nenu.edu.cn}
\cortext[cor2]{Corresponding author.}
\address[rvt1]{School of Mathematics and Statistics, Yili Normal University, Yining, 835000, China}
\address[rvt2]{School of Mathematics and Statistics, Northeast Normal University, Changchun, 130024, China}

\title{The Chern-Simons action for perturbed Dirac triples}
\begin{abstract}
In \cite{PO1} and \cite{PO2}, Pfante defined a noncommutative Chern-Simons action for 3-summable spectral triples and computed the Chern-Simons action for $\mathrm{SU}_q(2)$ and the noncommutative 3-torus. In this paper,we compute the Chern-Simons action for well-know perturbed Dirac triples.
\end{abstract}
\end{frontmatter}
\section{Introduction}

Noncommutative geometry, as formulated by Connes \cite{co5}, extends the classical notion of a differentiable manifold to spaces described by algebras, offering a powerful framework for quantum physics. The central technical tool is the spectral triple $(\mathcal{A}, \mathcal{H}, D)$, consisting of an algebra $\mathcal{A}$, a Hilbert space $\mathcal{H}$, and a Dirac operator $D$.

%Spectral triples are the starting point for the study of noncommutative manifolds. One can think of them being a generalization of the notion of ordinary differential manifolds because every spin manifold without boundary can be encoded uniquely by its spectral triple.

It turns out that themain technical device in noncommutative geometry, a spectral triple, naturally gives rise to a gauge theory. A key ingredient in many topological gauge theories is the Chern–Simons action. The Chern–Simons form was first introduced in \cite{CS} as a boundary term when the authors were computing the first Pontryagin number of a 4-manifold. It can be defined as a secondary characteristic class
by the transgression of the Chern character on principal bundles.
the noncommutative Chern-Simons action is the classical one 
\begin{align}\label{1.1}
	S_{\mathrm{CS}}(A):=\frac{k}{4 \pi} \int_{M} \operatorname{Trace}\left(A \wedge d A+\frac{2}{3} A \wedge A \wedge A\right),
\end{align}
where $A = A_{\mu} dx^{\mu}$ is Chern–Simons gauge ﬁeld and $k$ is called level of the theory. 
The classical Chern-Simons form is a topological invariant in the sense that it is independent of the background metric.
This classical Chern–Simons invariant has wide applications in differential geometry, global analysis, topology, and theoretical physics.

A noncommutative Chern–Simons action for 3‑summable spectral triples was introduced by Pfante in \cite{PO2}. It is gauge invariant modulo a Fredholm index that originates from the local index formula \cite{Co}. In this construction, the action includes not only a 3‑cocycle $\varphi_3$ but also a 1‑cocycle $\varphi_1$, and the pair $(\varphi_1,\varphi_3)$ forms a $(b,B)$-cocycle. When $\varphi_1$ vanishes, the action reduces to the definition given by Connes and Chamseddine. Pfante evaluated the Chern–Simons action on two explicit examples: the quantum compact group $\mathrm{SU}_q(2)$ \cite{PO2} and the noncommutative 3‑torus $C^{\infty}(\mathrm{T}^3_{\theta})$ \cite{PO1}. The $\varphi_1$ term yields a non‑trivial contribution in the $\mathrm{SU}_q(2)$ case, whereas it disappears on the 3‑torus.
In \cite{CF}, the Chern–Simons action is defined by employing a Chern–Simons form whose construction in cyclic cohomology is due to Quillen \cite{Qu}.

Considerable attention has also been given to the Dirac operator with torsion, the Dirac operator with inner fluctuations, and their respective perturbed Dirac triples.
Bismut \cite{Bismut} proved a local index theorem for Dirac operators on a Riemannian manifold $M$ associated with connections on $TM$ which have non zero torsion.
In \cite{Ac2}, Ackermann and Tolksdorf proved a generalized version of the well-known Lichnerowicz formula for the square of the most general Dirac operator with torsion $D_T$. This operator lives on an even-dimensional spin manifold and is associated to a metric connection with torsion.
In \cite{co4}, Connes and Chamseddine proved in the general framework
of noncommutative geometry that the inner fluctuations of the spectral action can be computed as residues and give exactly the counterterms for the Feynman graphs with fermionic internal lines, and showed that for geometries of dimension less than or equal to four the obtained terms add up to a sum of a Yang-Mills action with a Chern-Simons action.

In \cite{PO2}, Pfante shows that the Chern–Simons action is not a topological invariant. It generally depends on the particular spectral triple  $(\mathcal{A}, \mathcal{H}, D)$. It is natural to consider other spectral triples. Motivated by the Chern–Simons action \cite{PO2} and the perturbed Dirac triples \cite{co4, WWW}, the purpose of this paper is to generalize the results in \cite{PO1, PO2} and get some new Chern–Simons actions which is the extension of Chern–Simons actions to the perturbed Dirac triples. Our main theorems are as follows. 

%Pfante \cite{PO2} shows that the noncommutative Chern–Simons action is the classical one \eqref{1.1} if the spectral triple  $\left(C^{\infty}(M), L^{2}(S), D\right)$  comes from a 3-dimensional, closed spin manifold $M$ with gauge group  $\operatorname{SU}(N)$ .

%If $(\mathcal{A}, \mathcal{H}, \mathcal{D})$ is a spectral triple, then  $\left(\mathrm{M}_{N}(\mathcal{A}), \mathcal{H} \otimes \mathbb{C}^{N}, \mathcal{D} \otimes I_{N}\right)$ is a spectral triple over the matrix algebra  $\mathrm{M}_{N}(\mathcal{A})$  of $\mathcal{A}$, for a positive integer  $N \in \mathbb{N}$, satisfying all requirements imposed to $(\mathcal{A}, \mathcal{H}, \mathcal{D})$.
% and on the choice of Dirac operators

%Considering Dirac operators with torsion, Hanisch et al. \cite{Ha} derived a formula for the gravitational part of the spectral action for Dirac operators on 4-dimensional manifolds with totally anti-symmetric torsion, while Pf${\mathrm{\ddot{a}}}$ffle and Stephan \cite{pf3} gave a Lichnerowicz-type formula for the Dirac operator with torsion.

%we compute the Chern–Simons action with respect to these spectral triples

\begin{thm}\label{thm1}
	The Chern–Simons action on the smooth compact Riemannian 3-dimensional manifold $M$
	with respect to the spectral triple $\big( C^{\infty}(M) \otimes M_{N}(C), L^2(M, S(TM)) \otimes C^N, \tilde{D} \otimes Id_{N}\big)$ is given by
	\begin{align}
		S_{\mathrm{CS}}(A) = &\frac{k}{4\pi} c_{0} \int_{M} \operatorname{tr} \bigg(A  \wedge d A+\frac{2}{3} A \wedge A  \wedge A \bigg) \\
		&-8\sqrt{-1}\pi^{2} k \int_{M} \operatorname{tr}\big[ a^{0}\left\langle e_{1}^{*} \wedge e_{2}^{*}\wedge e_{3}^{*} , da^1 \wedge \nabla X^{*} \right\rangle \big] d{\rm Vol}_M\nonumber
	\end{align}
	for $\tilde{D}=D+\sqrt{-1}c(X)$ , $A=a^{0}d a^{1} $, and $a^{0}, a^{1} \in M_{N}\big(C^{\infty}(M)\big).$
\end{thm}

\begin{thm}\label{thm2}
	The Chern–Simons action on the smooth compact Riemannian 3-dimensional manifold $M$
	with respect to the spectral triple $\big( C^{\infty}(M) \otimes M_{N}(C), L^2(M, S(TM)) \otimes C^N, D_T \otimes Id_{N}\big)$ is given by
	\begin{align}
		S_{\mathrm{CS}}(A) = &\frac{k}{4\pi} c_{0} \int_{M} \operatorname{tr} \bigg(A  \wedge d A+\frac{2}{3} A \wedge A  \wedge A \bigg) 
	\end{align}
	for $D_T=D+fc(e_1)c(e_2)c(e_3)$ , $A=a^{0}d a^{1} $, and $a^{0}, a^{1} \in M_{N}\big(C^{\infty}(M)\big)$, where $\big( A, \operatorname{grad} f \big)$ denotes the pair between the one-form and vector fields.
\end{thm}

The paper is organized in the following way. In section 2, we first recall the definition of the Chern–Simons action \cite{PO2}. In Sections 3 and 4, we obtain two different spectral triples via the Dirac operator with inner fluctuations and Dirac operators with torsion, and we compute the Chern–Simons action with respect to these spectral triples.

\section{The Chern–Simons action}
To introduce the concept of the Chern–Simons action in noncommutative geometry, we first present three additional constraints and then give its general definition. see \cite{PO2} for a full discussion.

$\mathbf{ a) Dimension.}$  There is an integer $n$ such that the decreasing sequence $(\lambda_{k})_{k \in \mathbb{N}}$ of eigenvalues of the compact operator $|\mathcal{D}|^{-1}$ satisfies
\begin{align}
	\lambda_{k}=O\big(k^{-1 / n}\big)\nonumber
\end{align}
when $ k \rightarrow \infty$ .
The smallest integer which fulfils this condition is called the dimension of the spectral triple. In the case of a $p$-dimensional, closed spin manifold, the dimension of the triple $\left(C^{\infty}(M), L^{2}(S), \mathcal{D}\right)$ coincides with the dimension $p$ of the manifold.

$\mathbf{ b) Regularity.}$ Any element $b$ of the algebra generated by $\pi(\mathcal{A})$  and  $[\mathcal{D}, \pi(\mathcal{A})]$  is contained in the domain of  $\delta^{k}$  for all  $k \in \mathbb{N}$, i.e., $\delta^{k}(b)$ is densely defined and has a bounded extension on $\mathcal{H}$.

Oliver Pfante \cite{PO2} introduce the notion of the dimension spectrum of a spectral triple. This is the set $\Sigma \subset \mathbb{C}$ of singularities of functions
\begin{align}
\zeta_{b}(z)=\operatorname{Trace}\left(b|\mathcal{D}|^{-2 z}\right), \quad \operatorname{Re} z>p / 2, b \in \mathcal{B}.\nonumber
\end{align}

$\mathbf{ c) Dimension \; spectrum.}$ $\Sigma$ is a discrete subset of $\mathbb{C}$. Therefore, for any element $b$ of the algebra $\mathfrak{B}$ the functions
\begin{align}
\zeta_{b}(z)=\operatorname{Trace}\left(b|\mathcal{D}|^{-2 z}\right)\nonumber
\end{align}
extend holomorphically to  $\mathbb{C} \backslash \Sigma$.

\begin{defn}\textup{\cite{PO2}}\label{CSdefn}
 Let $( \mathcal{A}, \mathcal{H}, \mathcal{D} )$ be a spectral triple satisfying the dimension, regularity and Dimension spectrum conditions, $ N \in \mathbb{N}$ , and  $A \in \mathrm{M}_{N}\left(\Omega^{1}(\mathcal{A})\right)$ a hermitian matrix of 1-forms. We define a Chern-Simons action by

\begin{align}
S_{\mathrm{CS}}(A)=6 \pi k \phi_{3}\left(A d A+\frac{2}{3} A^{3}\right)-2 \pi k \phi_{1}(A)
\end{align}
for an integer k, the cyclic cocycle  $\left(\phi_{3}, \phi_{1}\right) $ as follows

\begin{align}
	&\varphi_{3}\left(a^{0}, a^{1}, a^{2}, a^{3}\right)=\frac{1}{12} \int\hspace{-1.05em}- a^{0} \left[\mathcal{D}, a^{1}\right] \left[\mathcal{D}, a^{2}\right] \left[\mathcal{D}, a^{3}\right]|\mathcal{D}|^{-3},\\
	&\varphi_{1}\left(a^{0}, a^{1}\right)=\int\hspace{-1.05em}- a^{0}\left[\mathcal{D}, a^{1}\right]|\mathcal{D}|^{-1} -\frac{1}{4} \int\hspace{-1.05em}- a^{0} \nabla\left(\left[\mathcal{D}, a^{1}\right]\right)|\mathcal{D}|^{-3} +\frac{1}{8} \int\hspace{-1.05em}- a^{0} \nabla^{2}\left(\left[\mathcal{D}, a^{1}\right]\right)|\mathcal{D}|^{-5},
\end{align}
where $a^{0}, a^{1}, a^{2}, a^{3} \in \mathcal{A} $ and $\nabla(T)=\left[\mathcal{D}^2, T\right]$ for any $T$ in $\mathcal{A}$ or $\left[\mathcal{D}, \mathcal{A}\right]$.

\end{defn}

%Consider spectral triple $\big( C^{\infty}(M) \otimes M_{N}(C), (D+\sqrt{-1}C(X)) \otimes Id_{N}, L^2(M, S(TM)) \otimes C^N \big)$ and $\big( C^{\infty}(M) \otimes M_{N}(C), L^2(M, S(TM)) \otimes C^N, D_T \otimes Id_{N}\big)$, $A \in \Omega^{1}\big( C^{\infty}(M) \otimes M_{N}(C) \big)$, $ A=A_1 d A_2 $, $A_1, A_2 \in M_{N}\big(C^{\infty}(M)\big) $, 

\section{The Dirac operator with inner fluctuations}
We give some definitions and basic notions which we will use in this paper.

Let $M$ be a smooth compact oriented Riemannian $n$-dimensional manifolds without boundary and $N$ be a vector bundle on $M$.
We say that $P$ is a differential operator of Laplace type, if it has locally the form
\begin{equation}\label{p}
	P=-(g^{ij}\partial_i\partial_j+A^i\partial_i+B),
\end{equation}
where $\partial_{i}$  is a natural local frame on $TM,$ $(g^{ij})_{1\leq i,j\leq n}$ is the inverse matrix associated to the metric
matrix  $(g_{ij})_{1\leq i,j\leq n}$ on $M,$ $A^{i}$ and $B$ are smooth sections of $\textrm{End}(N)$ on $M$ (endomorphism).
If $P$ satisfies the form \eqref{p}, then there is a unique
connection $\nabla$ on $N$ and a unique endomorphism $E$ such that
\begin{equation}
	P=-[g^{ij}(\nabla_{\partial_{i}}\nabla_{\partial_{j}}- \nabla_{\nabla^{L}_{\partial_{i}}\partial_{j}})+E],\nonumber
\end{equation}
where $\nabla^{L}$ is the Levi-Civita connection on $M$. Moreover
(with local frames of $T^{*}M$ and $N$), $\nabla_{\partial_{i}}=\partial_{i}+\omega_{i} $
and $E$ are related to $g^{ij}$, $A^{i}$ and $B$ through
\begin{eqnarray}
	&&\omega_{i}=\frac{1}{2}g_{ij}\big(A^{i}+g^{kl}\Gamma_{ kl}^{j} id\big),\nonumber\\
	&&E=B-g^{ij}\big(\partial_{i}(\omega_{j})+\omega_{i}\omega_{j}-\omega_{k}\Gamma_{ ij}^{k} \big),\nonumber
\end{eqnarray}
where $\Gamma_{ kl}^{j}$ is the  Christoffel coefficient of $\nabla^{L}$.

Now we let $M$ be a $n$-dimensional oriented spin manifold
with Riemannian metric $g$. The Dirac operator $D$ is
locally given as follows in terms of orthonormal frames $e_i,~1\leq
i\leq n$ and natural frames $\partial_i$ of $TM$, one has
\begin{align}
	D=\sum_{i,j}g^{ij}c(\partial_i)\nabla^S_{\partial_j}=\sum_{i}c(e_i)\nabla^S_{e_i},
\end{align}
where $c(e_i)$ denotes the Clifford action which satisfies the relation
$$c(e_i)c(e_j)+c(e_j)c(e_i)=-2\delta_i^j,$$ and
\begin{align}
	\nabla^S_{\partial_i}=\partial_i+\sigma_i,~~\sigma_i=\frac{1}{4}\sum_{j,k}\left<\nabla^L_{\partial_i}e_j,e_k\right>c(e_j)c(e_k).
\end{align}
Let
\begin{align}
	\partial^j=g^{ij}\partial_i,~~\sigma^i=g^{ij}\sigma_j,~~\Gamma^k=g^{ij}\Gamma_{ij}^k.
\end{align}
Recall the Lichnerowicz formula for the square of the Dirac operator, by (6a) in  \cite{Ka}, we have
\begin{align}
	D^2=-g^{ij}\partial_i\partial_j-2\sigma^j\partial_j+\Gamma^k\partial_k-g^{ij}[\partial_i(\sigma_j)+\sigma_i\sigma_j-\Gamma_{ij}^k\sigma_k]+\frac{1}{4}s,
\end{align}

we have respective symbols of $D^{2}$:

\begin{lem}\cite{Ka}\label{lemma1}
	\begin{align}
		\sigma_2(D^2)&=|\xi|^2,\\
		\sigma_1(D^2)&=\sqrt{-1}\sum_{\mu=1}^{n}(\Gamma^\mu-2\sigma^\mu)\xi_\mu,\\
		\sigma_0(D^2)&=-\sum_{\mu,\nu=1}^{n}g^{\mu\nu}(\partial^{x}_\mu\sigma_\nu+\sigma_\mu\sigma_\nu-\Gamma^\alpha_{\mu\nu}\sigma_\alpha)+\frac{1}{4}s,
	\end{align}
where $s$ is the scalar curvature.
\end{lem}

\begin{lem}\cite{DL3}\label{lemma2}
 For $\mathbb{D}=D+\mathcal{B}$, The homogeneous symbols  $\sigma\left(|\mathbb{D}|^{k}\right)=\sigma_{k}(|\mathbb{D}|^{k})+\sigma_{k-1}(|\mathbb{D}|^{k})+\sigma_{k-2}(|\mathbb{D}|^{k})+\ldots $ read 
\begin{align}
\sigma_{k}(|\mathbb{D}|^{k}) & =|\xi|^{k-2}\sum_{a,b,c,d=1}^{n}\left(\delta_{a, b}+\frac{k}{6} R_{a c b d} x^{c} x^{d}\right) \xi_{a} \xi_{b}+o\left(\mathbf{x}^{2}\right),\\
\sigma_{k-1}(|\mathbb{D}|^{k}) & = \sqrt{-1} k|\xi|^{k-2}\sum_{a,b,j,k=1}^{n} \xi_{a} x^{b}\left(\frac{1}{3} \operatorname{Ric}_{a b}-\frac{1}{8} R_{a b j k}c(e_j)c(e_k)\right)\\
&+\frac{\sqrt{-1}}{2}k|\xi|^{k-2}\sum_{a=1}^{n} \xi_{a}\left\{c(e_a), \mathcal{B}\right\}+o(\mathbf{x}), \nonumber\\
\sigma_{k-2}(|\mathbb{D}|^{k}) & =\frac{k}{8}|\xi|^{k-2} R+\frac{k(k-2)}{12}|\xi|^{k-4}\sum_{a,b=1}^{n} \operatorname{Ric}_{a b} \xi_{a} \xi_{b}+\frac{k}{2}|\xi|^{k-2}\sum_{a=1}^{n}\left(c(e_a) \mathcal{B}_{a}+\mathcal{B}_{0}^{2}\right) \\
& +\frac{k(k-2)}{4}|\xi|^{k-4}\sum_{a,b=1}^{n} \xi_{a} \xi_{b}\left(\left\{c(e_a), \mathcal{B}_{b}\right\}-\frac{1}{2}\left\{c(e_a), \mathcal{B}_{0}\right\}\left\{c(e_b), \mathcal{B}_{0}\right\}\right)+o(\mathbf{1}),\nonumber
\end{align}
where the expansions in normal coordinates were used: $\mathcal{B}=\mathcal{B}_0+\mathcal{B}_a x^a+o(\mathbf{x})$.
\end{lem}

Let $X$ be a
vector field on $M$, and we also denote the associated
Clifford action by $\sqrt{-1}c(X)$. For Dirac operator with inner fluctuations $\tilde{D} :=D+\sqrt{-1}c(X)$, we have respective symbols of $\tilde{D}^{2}$:
\begin{lem}\label{lemma3.3}

	\begin{align}
		\sigma_1(\tilde{D}^2)=&\sqrt{-1}\sum_{\mu=1}^{n}(\Gamma^\mu-2\sigma^\mu)\xi_\mu+2\sum_{\mu=1}^{n}g(X,e_{\mu})(x)\xi_\mu,\label{lem3.3.1}\\
		\sigma_0(\tilde{D}^2)=&-\sum_{\mu,\nu=1}^{n}g^{\mu\nu}(\partial^{x}_\mu\sigma_\nu+\sigma_\mu\sigma_\nu-\Gamma^\alpha_{\mu\nu}\sigma_\alpha+\sqrt{-1}c(X)c(e_\nu) \sigma_\nu+\sqrt{-1} c(e_\nu)\sigma_\nu c(X))\nonumber\\
		&+\sqrt{-1}\sum_{\mu,\gamma=1}^{n}\partial^{x}_\mu(X_\gamma)c(e_\mu)c(e_\gamma) -c^2(X)+\frac{1}{4}s.\label{lem3.3.2}
	\end{align}
\end{lem}
\begin{proof}
Recall the Lichnerowicz formula for the square of the Dirac operator, by (6a) in  \cite{Ka}, we have
	\begin{align}
	D^2=-g^{ij}\partial_i\partial_j-2\sigma^j\partial_j+\Gamma^k\partial_k-g^{ij}[\partial_i(\sigma_j)+\sigma_i\sigma_j-\Gamma_{ij}^k\sigma_k]+\frac{1}{4}s,
\end{align}
where $s$ is the scalar curvature. 

For Dirac operator with inner fluctuations $\tilde{D} :=D+\sqrt{-1}c(X)$, then
\begin{align}\label{sigma0}
	\tilde{D}^2=D^2+\sqrt{-1}Dc(X)+\sqrt{-1}c(X) D-c^2(X).
\end{align}
Using lemma \ref{lemma1} and $\sigma(P_1 P_2)=\sum_{|\alpha|=0}^\infty\frac{(-i)^{|\alpha|}}{\alpha!}\partial^\alpha_\xi[\sigma(P_1)]\partial^\alpha_x[\sigma(P_2)]$, we compute
\begin{align}
	\sigma_1(\tilde{D}^2)=&\sigma_1(D^2)+\sqrt{-1}\sigma_1(D)c(X)+\sqrt{-1}c(X)\sigma_1(D)\nonumber\\
	=&\sqrt{-1}\sum_{\mu=1}^{n}(\Gamma^\mu-2\sigma^\mu)\xi_\mu+2\sum_{\mu=1}^{n}g(X,e_{\mu})(x)\xi_\mu,\\
	\sigma_0(\tilde{D}^2)=&\sigma_0(D^2)+\sqrt{-1}\sigma_0(D)c(X)+\sqrt{-1}c(X)\sigma_0(D)\nonumber\\
	&+(-\sqrt{-1})\sum_{\mu=1}^{n}\partial_{\xi_\mu}\sigma_1(D)\partial_{x_\mu}(\sqrt{-1}c(X))-c^2(X)\nonumber\\
	=&-\sum_{\mu,\nu=1}^{n}g^{\mu\nu}(\partial^{x}_\mu\sigma_\nu+\sigma_\mu\sigma_\nu-\Gamma^\alpha_{\mu\nu}\sigma_\alpha+\sqrt{-1}c(X)c(e_\nu) \sigma_\nu+\sqrt{-1} c(e_\nu)\sigma_\nu c(X))\nonumber\\
	&+\sqrt{-1}\sum_{\mu,\gamma=1}^{n}\partial^{x}_\mu(X_\gamma)c(e_\mu)c(e_\gamma) -c^2(X)+\frac{1}{4}s.
\end{align}

\end{proof}

\begin{lem}\label{lemma3.4}
	In normal coordinates around a fixed point of the manifold $M$, the symbols representation of the inverse of the Dirac operator with inner fluctuations read:
	\begin{align} 
		\sigma_{-3}(|\tilde{D}|^{-1}) 
		=&-\frac{1}{8}|\xi|^{-3} R +\frac{1}{4}|\xi|^{-5}\sum_{a,b=1}^{n}  \operatorname{Ric}_{a b}  \xi_{a} \xi_{b}\nonumber\\
		&-\frac{\sqrt{-1}}{2}|\xi|^{-3}\sum_{a,l=1}^{n}\partial_{x_a}(X_l) c(e_a) c(e_l)+\frac{1}{2}|\xi|^{-3}c^{2}(X) \nonumber\\
		&-\frac{3}{2}\sqrt{-1}|\xi|^{-5} \sum_{a,b=1}^{n}\sqrt{-1}\partial_{x_b}(X_a)\xi_{a} \xi_{b}+\frac{3}{2}|\xi|^{-5}\sum_{a,b=1}^{n}g(X,e_a)g(X,e_b)\xi_{a} \xi_{b}++o(\mathbf{1}),\label{lem3.4.1}\\
		\sigma_{-3}(|\tilde{D}|^{-3})=&|\xi|^{-5}\sum_{a,b,c,d=1}^{n} \left(\delta_{a b}-\frac{1}{2} R_{a c b d} x^{c} x^{d}\right) \xi_{a} \xi_{b}+O\left(\mathbf{x}^{2}\right),\label{lem3.4.2} \\
		\sigma_{-4}(|\tilde{D}|^{-3})=&-\sqrt{-1}|\xi|^{-5}\sum_{a,b=1}^{n}  \operatorname{Ric}_{a b} x^{b} \xi_{a}\label{lem3.4.3}\nonumber\\
		&+\frac{ 3 \sqrt{-1}}{8}|\xi|^{-5}\sum_{a,b,j,k=1}^{n} \operatorname{R}_{b a j k}(x_0)c(e_j)c(e_k) x^{b} \xi_{a}\nonumber\\
		&+\frac{3}{2}|\xi|^{-5}\sum_{a=1}^{n}\{c(e_a),c(X)\} \xi_a +O\left(\mathbf{x}\right),\\
		\sigma_{-5}(|\tilde{D}|^{-5})=&|\xi|^{-7}\sum_{a,b,c,d=1}^{n} \left(\delta_{a b}-\frac{5}{6} R_{a c b d} x^{c} x^{d}\right) \xi_{a} \xi_{b}+O\left(\mathbf{x}^{2}\right),\label{lem3.4.4}
	\end{align}
	where ${\rm R}_{a c b d}$ and $\operatorname{Ric}_{a b}$ are the components of the Riemann and Ricci tensor, s is the scalar curvature.
\end{lem}

%\begin{lem}
%The homogeneous symbols $ \sigma\left(\left[|\tilde{D}|^{k}, v_{b} c(e_b)\right]\right)=\sigma_{k}+\sigma_{k-1}+\sigma_{k-2}+\ldots $  read

%\begin{align}
%	\sigma_{k} =& o\left(\mathbf{x}^{2}\right) \\
%	\sigma_{k-1} =&\; \sqrt{-1} \frac{k}{2}|\xi|^{k-2} \xi_{c} v_{b} x^{d} R_{c d j b} c(e_j)-\sqrt{-1} k \xi_{c}|\xi|^{k-2} v_{b c} c(e_b)-\frac{k}{2}|\xi|^{k-2} v_{b}\left[\left\{c(e_c), B\right\}, c(e_b)\right]+o(\mathbf{x}), \\
%\sigma_{k-2}= & -k v_{b c c} c(e_b)|\xi|^{k-2}-k(k-2) \xi_{c} \xi_{d}|\xi|^{k-4} v_{b c d} c(e_b)+\frac{k}{2}|\xi|^{k-2} v_{b}\left[\sqrt{-1} c(e_c) B_{c}+B_{0}^{2}, c(e_b)\right] \nonumber\\
%	& +\frac{k(k-2)}{4}|\xi|^{k-4} \xi_{c} \xi_{d} v_{b}\left[\sqrt{-1}\left\{c(e_c), B_{d}\right\}+\frac{1}{2}\left\{c(e_c), B_{0}\right\}\left\{c(e_d), B_{0}\right\}, c(e_b)\right]\nonumber \\
%	& +\sqrt{-1} \frac{k}{2}|\xi|^{k-2}\left\{c(e_c), B_{0}\right\} v_{b c} c(e_b)+\sqrt{-1} \frac{k(k-2)}{2}|\xi|^{k-4} \xi_{c} \xi_{d}\left\{c(e_c), B_{0}\right\} v_{b d} c(e_b)+o(\mathbf{1}),
%\end{align}
%\end{lem}

\begin{lem}\label{lemma3.5}
	
	  \begin{align} 
		\sigma_{1} \bigg(	\nabla\left(\left[\tilde{D}, a^{1}\right]\right)\bigg)(x_0)=&\;-2\sqrt{-1}\sum_{j,\gamma=1}^{n}\partial_{x_j}[e_\gamma(a^1)]c(e_\gamma)\xi_j,\label{lem3.5.1}\\
		\sigma_{0} \bigg(	\nabla\left(\left[\tilde{D}, a^{1}\right]\right)\bigg)(x_0)=&\;\sqrt{-1}\sum_{j,\gamma=1}^{n}\partial_{x_j}(X_\gamma)[c(e_j)c(e_\gamma)c(da^1)-c(da^1)c(e_j)c(e_\gamma)]\label{lem3.5.2}\\
		&-2\sqrt{-1}\sum_{j,\gamma=1}^{n}g(e_j,X)\partial_{x_j}[e_\gamma(a^1)]c(e_\gamma)-\sum_{j,\gamma=1}^{n}\partial_{x_j}^{2}[e_\gamma(a^1)]c(e_\gamma),\nonumber\\
		\sigma_{2} \left(	\nabla^2\left(\left[\tilde{D}, a^{1}\right]\right)\right)(x_0)=
		&\;\frac{1}{2}\sum_{i,j,s,t,l=1}^{n}R_{jlst}(x_0)c(e_s)c(e_t)c(da^1)\xi_j\xi_l\label{lem3.5.3}\\
		&-\frac{1}{2}\sum_{i,j,s,t,l=1}^{n}R_{jlst}(x_0)c(da^1)c(e_s)c(e_t)\xi_j\xi_l\nonumber\\
		&-4\sum_{j,l,\gamma=1}^{n}\partial_{x_j}\partial_{x_l}[e_\gamma(a^1)](x_0)c(e_\gamma)\xi_j\xi_l.\nonumber
	\end{align}  	
\end{lem}

\begin{proof}
	
We start with the observation that
 \begin{align} 
 	\nabla\left(\left[\tilde{D}, a^{1}\right]\right)=\left[\tilde{D}^{2}, c(da^1)\right]= \tilde{D}^{2} \circ c(da^1)- c(da^1) \circ \tilde{D}^{2}.
 \end{align}

Next, the principal symbol of a product of pseudo-differential operators in terms of the principal symbols of the factors, namely:
 \begin{align} 
\sigma\bigg( \tilde{D}^{2} \circ c(da^1)\bigg)=&\sum_{|\beta|=0}^\infty\frac{(-\sqrt{-1})^{|\beta|}}{\beta!}\partial^\beta_\xi(\sigma(\tilde{D}^{2}))\partial^\beta_x[\sigma(c(da^1))\label{le5.1}\\
=&\sigma_2(\tilde{D}^{2})c(da^1)+\sigma_1(\tilde{D}^{2})c(da^1)+(-\sqrt{-1})\sum_{j=1}^{n}\partial_{\xi_j}(\sigma_2(\tilde{D}^{2}))\partial_{x_j}(c(da^1))\nonumber\\
&+\sigma_0(\tilde{D}^{2})c(da^1)+(-\sqrt{-1})\sum_{j=1}^{n}\partial_{\xi_j}(\sigma_1(\tilde{D}^{2}))\partial_{x_j}(c(da^1))\nonumber\\
&-\sum_{j,l=1}^{n}\partial_{\xi_j}\partial_{\xi_l}(\sigma_2(\tilde{D}^{2}))\partial_{x_j}\partial_{x_l}(c(da^1)),\nonumber\\
\sigma\bigg( c(da^1)  \circ \tilde{D}^{2} \bigg)=&\sum_{|\beta|=0}^\infty\frac{(-\sqrt{-1})^{|\beta|}}{\beta!}\partial^\beta_\xi[\sigma(c(da^1))]\partial^\beta_x[\sigma(\tilde{D}^{2})]\label{le5.2}\\
	=&c(da^1)\sigma_2(\tilde{D}^{2})+c(da^1)\sigma_1(\tilde{D}^{2})+c(da^1)\sigma_0(\tilde{D}^{2}),\nonumber
\end{align}
According to \eqref{le5.1} and \eqref{le5.2}, we get
 \begin{align} 
	\sigma\bigg( \left[\tilde{D}^{2}, c(da^1)\right]\bigg)=&\sigma\bigg( \tilde{D}^{2} \circ c(da^1)\bigg)-\sigma\bigg( c(da^1) \circ \tilde{D}^{2}\bigg)\nonumber\\
	=&\sigma_1(\tilde{D}^{2})c(da^1)-c(da^1)\sigma_1(\tilde{D}^{2})+(-\sqrt{-1})\sum_{j=1}^{n}\partial_{\xi_j}(\sigma_2(\tilde{D}^{2}))\partial_{x_j}(c(da^1))\nonumber\\
	&+\sigma_0(\tilde{D}^{2})c(da^1)-c(da^1)\sigma_0(\tilde{D}^{2})+(-\sqrt{-1})\sum_{j=1}^{n}\partial_{\xi_j}(\sigma_1(\tilde{D}^{2}))\partial_{x_j}(c(da^1))\nonumber\\
	&-\sum_{j,l=1}^{n}\partial_{\xi_j}\partial_{\xi_l}(\sigma_2(\tilde{D}^{2}))\partial_{x_j}\partial_{x_l}(c(da^1)),
\end{align}

where
  \begin{align} 
	\sigma_{1}\bigg( \left[\tilde{D}^{2}, c(da^1)\right]\bigg)=&\sigma_1(\tilde{D}^{2})c(da^1)-c(da^1)\sigma_1(\tilde{D}^{2})+(-\sqrt{-1})\sum_{j=1}^{n}\partial_{\xi_j}(\sigma_2(\tilde{D}^{2}))\partial_{x_j}(c(da^1))\label{2.29}\\
	=&-2\sqrt{-1}\sigma^\mu c(da^1)\xi_\mu +2\sqrt{-1}c(da^1)\sigma^\mu \xi_\mu-2\sqrt{-1}\sum_{j,\gamma=1}^{n}\partial_{x_j}[e_\gamma(a^1)]c(e_\gamma)\xi_j, \nonumber\\
	\sigma_{0} \bigg( \left[\tilde{D}^{2}, c(da^1)\right]\bigg)=&\sigma_0(\tilde{D}^{2})c(da^1)-c(da^1)\sigma_0(\tilde{D}^{2})+(-\sqrt{-1})\sum_{j=1}^{n}\partial_{\xi_j}(\sigma_1(\tilde{D}^{2}))\partial_{x_j}(c(da^1))\label{2.30}\\
	&-\sum_{j,l=1}^{n}\partial_{\xi_j}\partial_{\xi_l}(\sigma_2(\tilde{D}^{2}))\partial_{x_j}\partial_{x_l}(c(da^1))\nonumber\\
	=&-\sum_{\mu,\nu=1}^{n}g^{\mu\nu}(\partial^{x}_\mu\sigma_\nu+\sigma_\mu\sigma_\nu-\Gamma^\alpha_{\mu\nu}\sigma_\alpha+\sqrt{-1}c(X)c(e_\nu) \sigma_\nu+\sqrt{-1} c(e_\nu)\sigma_\nu c(X))c(da^1)\nonumber\\
	&+\sqrt{-1}\sum_{\mu,\gamma=1}^{n}\partial^{x}_\mu(X_\gamma)c(e_\mu)c(e_\gamma)c(da^1)-\sqrt{-1}\sum_{\mu,\gamma=1}^{n}\partial^{x}_\mu(X_\gamma)c(da^1)c(e_\mu)c(e_\gamma) \nonumber\\
	&+\sum_{\mu,\nu=1}^{n}g^{\mu\nu}c(da^1)(\partial^{x}_\mu\sigma_\nu+\sigma_\mu\sigma_\nu-\Gamma^\alpha_{\mu\nu}\sigma_\alpha+\sqrt{-1}c(X)c(e_\nu) \sigma_\nu+\sqrt{-1} c(e_\nu)\sigma_\nu c(X))\nonumber\\
	&+\sum_{j,\gamma=1}^{n}(\Gamma^j-2\sigma^j)\partial_{x_j}[e_\gamma(a^1)]c(e_\gamma)-2\sqrt{-1}\sum_{j,\gamma=1}^{n}g(X,e_{j})(x)\partial_{x_j}[e_\gamma(a^1)]c(e_\gamma)\nonumber\\
    &-\sum_{j,\gamma=1}^{n}\partial_{x_j}^{2}[e_\gamma(a^1)]c(e_\gamma).\nonumber
\end{align}

Using the facts:
\begin{align}
	&\Gamma^\mu(x_0)=0,\label{2.31}\\
	&\sigma^\mu(x_0)=0,\label{2.32}\\
	&\partial_{x_k}(\Gamma^\mu)(x_0)=g^{\alpha\beta}\partial_{x_k}(\Gamma_{\alpha\beta}^\mu)(x_0)=\frac{2}{3}R_{k\alpha\mu\alpha}(x_0),\label{2.33}\\
	&\partial_{x_k}(\sigma^\mu)(x_0)=\frac{1}{4}\sum_{st=1}^n\partial_{x_k}(<\nabla^L_{\partial_\mu }e_s,e_t>)(x_0)c(e_s)c(e_t)=\frac{1}{8}\sum_{st=1}^nR_{k\mu t s}(x_0)c(e_s)c(e_t),\label{2.34}
\end{align}
we get \eqref{lem3.5.1} and \eqref{lem3.5.2} in lemma \ref{lemma3.5}.

Based on $\nabla(T)=\left[D^2, T\right]$, the following equation is obtained:
 \begin{align} 
 	\nabla^2\left(\left[\tilde{D}, a^{1}\right]\right)=\left[\tilde{D}^{2}, [\tilde{D}^{2}, c(da^1)]\right],
 \end{align}
let $\mathcal{P}_{1}:=\nabla\left(\left[\tilde{D}, a^{1}\right]\right)$, then
  \begin{align} 
 	\nabla^2\left(\left[\tilde{D}, a^{1}\right]\right)= \tilde{D}^{2} \circ \mathcal{P}_{1}- \mathcal{P}_{1} \circ \tilde{D}^{2}.
 \end{align}
 By calculating the principal symbol of a product of pseudo-differential operators, we obtain
   \begin{align} 
 	\sigma_{2}\bigg(\tilde{D}^{2} \circ \mathcal{P}_{1}\bigg)=&\sigma_2(\tilde{D}^{2})\sigma_0(\mathcal{P}_{1})+\sigma_1(\tilde{D}^{2})\sigma_1(\mathcal{P}_{1})+(-\sqrt{-1})\sum_{j=1}^{n}\partial_{\xi_j}(\sigma_2(\tilde{D}^{2}))\partial_{x_j}(\sigma_1(\mathcal{P}_{1})),\\ 
 	\sigma_{2}\bigg(\mathcal{P}_{1} \circ \tilde{D}^{2}\bigg)=&\sigma_1(\mathcal{P}_{1})\sigma_1(\tilde{D}^{2})+\sigma_0(\mathcal{P}_{1})\sigma_2(\tilde{D}^{2})+(-\sqrt{-1})\sum_{j=1}^{n}\partial_{\xi_j}(\sigma_1(\mathcal{P}_{1})\partial_{x_j}(\sigma_2(\tilde{D}^{2})).
 \end{align}
 
 According to \eqref{2.29} and \eqref{2.30}, we obtain
 \begin{align} 
 	\sigma_{2} \left(	\nabla^2\left(\left[\tilde{D}, a^{1}\right]\right)\right)=&\sigma_{2}\bigg(\tilde{D}^{2} \circ \mathcal{P}_{1}\bigg)-\sigma_{2}\bigg(\mathcal{P}_{1} \circ \tilde{D}^{2}\bigg)\nonumber\\
 	=&\sigma_1(\tilde{D}^{2})\sigma_1(\mathcal{P}_{1})-\sigma_1(\mathcal{P}_{1})\sigma_1(\tilde{D}^{2})+(-\sqrt{-1})\sum_{j=1}^{n}\partial_{\xi_j}(\sigma_2(\tilde{D}^{2}))\partial_{x_j}(\sigma_1(\mathcal{P}_{1})).
 \end{align}
 Using \eqref{2.31} to \eqref{2.34}, obtain
  \begin{align} 
 \bigg(\sigma_1(\tilde{D}^{2})\sigma_1(\mathcal{P}_{1})-\sigma_1(\mathcal{P}_{1})\sigma_1(\tilde{D}^{2})\bigg)(x_0)=0,
 \end{align}
and 
   \begin{align} 
 	&(-\sqrt{-1})\sum_{j=1}^{n}\partial_{\xi_j}(\sigma_2(\tilde{D}^{2}))\partial_{x_j}(\sigma_1(\mathcal{P}_{1}))(x_0)\nonumber\\
 	=&-\frac{1}{2}R_{j \mu t s}c(e_s)c(e_t)c(da^1)\xi_j\xi_\mu  +\frac{1}{2}R_{j \mu t s}c(da^1)c(e_s)c(e_t)\xi_j\xi_\mu \nonumber\\
 	&-4\sum_{j,l,\gamma=1}^{n}\partial_{x_j}\partial_{x_l}[e_\gamma(a^1)]c(e_\gamma)\xi_j\xi_l,
 \end{align}
 we get \eqref{lem3.5.3} in lemma \ref{lemma3.5}.
\end{proof}

To prove Theorem \ref{thm1}, based on Definition \ref{CSdefn}, we first compute $\phi_{1}(A)$, for $\tilde{D}=D+\sqrt{-1}c(X)$ , $A=a^{0}d a^{1} $, and $a^{0}, a^{1} \in M_{N}\big(C^{\infty}(M)\big),$ where

\begin{align}
	\phi_{1}\left(a^{0}, a^{1}\right)=&\int\hspace{-1.05em}- a^{0}\left[\tilde{D}, a^{1}\right]|\tilde{D}|^{-1}  -\frac{1}{4} \int\hspace{-1.05em}- a^{0} \nabla\left(\left[\tilde{D}, a^{1}\right]\right)|\tilde{D}|^{-3}  +\frac{1}{8} \int\hspace{-1.05em}- a^{0} \nabla^{2}\left(\left[\tilde{D}, a^{1}\right]\right)|\tilde{D}|^{-5}.
\end{align}

$\mathbf{Part~~i)}$
Let $n=3$, by \cite{FGV2}, we need to compute  $\int_{S^*M}{\rm tr}\bigg[\sigma_{-3}\bigg(a^{0}\left[\tilde{D}, a^{1}\right]|\tilde{D}|^{-1}\bigg)\bigg](x,\xi).$  Based on the algorithm yielding the principal
symbol of a product of pseudo-differential operators in terms of the principal symbols of the factors, we have
\begin{align}
	\sigma_{-3}\bigg(a^{0}\left[\tilde{D}, a^{1}\right]|\tilde{D}|^{-1}\bigg)(x_0)
	=&a^{0}c(da^1)\sigma_{-3}(|\tilde{D}|^{-1})(x_0)\nonumber\\
	=&-\frac{1}{8}a^{0}|\xi|^{-3} R c(da^1)+\frac{1}{4}\sum_{a,b=1}^{n} a^{0}|\xi|^{-5} \operatorname{Ric}_{a b} c(da^1) \xi_{a} \xi_{b}\nonumber\\
	&-\frac{\sqrt{-1}}{2}a^{0}|\xi|^{-3}\sum_{a,l=1}^{n}\partial_{x_a}(X_l)c(da^1) c(e_a) c(e_l)+\frac{1}{2}a^{0}|\xi|^{-3}c(da^1)c^{2}(X) \nonumber\\
	&-\frac{3}{2}\sqrt{-1}a^{0}|\xi|^{-5} \sum_{a,b=1}^{n}\partial_{x_b}(X_a)c(da^1)\xi_{a} \xi_{b}\nonumber\\
	&+\frac{3}{2}a^{0}|\xi|^{-5}\sum_{a,b=1}^{n}g(X,e_a)g(X,e_b)c(da^1)\xi_{a} \xi_{b}.
\end{align}

 Below, we compute each term of  $\int_{|\xi|=1} \operatorname{tr} \bigg[	\sigma_{-3}\bigg(a^{0}\left[\tilde{D}, a^{1}\right]|\tilde{D}|^{-1}\bigg)\bigg](x, \xi) \sigma(\xi)$ in turn.  
 When $n=3$, $\rm{tr}[id]=2 $ and $\rm{tr}\big( c(e_1) c(e_2) c(e_3) \big)=-2$. Then, based on the relation of the Clifford action $\operatorname{tr}\big(c(\mathcal{X})c(\mathcal{Y})\big)=-g(\mathcal{X},\mathcal{Y}) \rm{tr}[id]$, we get the following equations.

 \begin{align}
 &{\rm tr}\bigg(-\frac{1}{8}a^{0}|\xi|^{-3} R c(da^1)\bigg)\bigg|_{|\xi|=1}=\;0,\\
 &{\rm tr}\bigg(\frac{1}{4}\sum_{a,b=1}^{n}a^{0}|\xi|^{-5} \operatorname{Ric}_{a b} c(da^1) \xi_{a} \xi_{b}\bigg)\bigg|_{|\xi|=1}=\;0,\\
 &{\rm tr}\bigg(\frac{1}{2}a^{0}|\xi|^{-3}c^{2}(X)c(da^1)\bigg)\bigg|_{|\xi|=1}=\;0,\\
 &{\rm tr}\bigg(-\frac{3}{2}\sqrt{-1}a^{0}|\xi|^{-5} \sum_{a,b=1}^{n}\partial_{x_b}(X_a)c(da^1)\xi_{a} \xi_{b}\bigg)\bigg|_{|\xi|=1}=\;0,\\
 &{\rm tr}\bigg(-\frac{3}{2}a^{0}|\xi|^{-5}g(X,e_a)g(X,e_b)c(da^1)\xi_{a} \xi_{b}\bigg)\bigg|_{|\xi|=1}=\;0,\\
	&{\rm tr}\bigg(-\frac{\sqrt{-1}}{2}a^{0}|\xi|^{-3}\sum_{a,l=1}^{3}\partial_{x_a}(X_l)c(da^1) c(e_a) c(e_l)\bigg)\bigg|_{|\xi|=1}\\
	&=\sqrt{-1}a^{0}\sum_{a,l=1}^{3}\partial_{x_a}(X_l)\left\langle e_{1}^{*} \wedge e_{2}^{*}\wedge e_{3}^{*} , da^1 \wedge e_{a}^{*}\wedge e_{l}^{*} \right\rangle\nonumber\\
	&=\sqrt{-1}a^{0}\left\langle e_{1}^{*} \wedge e_{2}^{*}\wedge e_{3}^{*} , da^1 \wedge \nabla X^{*} \right\rangle\nonumber
\end{align}
Since $\int_{|\xi|=1} \xi_{a} \xi_{b}\sigma(\xi)=\frac{1}{n}\delta_{a}^{b}{\rm Vol}(S^{n-1})$, we obtain
 \begin{align}
&\int_{|\xi|=1}{\rm tr}\bigg(\sigma_{-3}\big(a^{0}\left[\tilde{D}, a^{1}\right]|\tilde{D}|^{-1}\big)(x_0)\bigg)\sigma(\xi)\nonumber\\
&=\sqrt{-1}a^{0}\left\langle e_{1}^{*} \wedge e_{2}^{*}\wedge e_{3}^{*} , da^1 \wedge \nabla X^{*} \right\rangle {\rm Vol}(S^{2})
\end{align}

$\mathbf{Part~~ii)}$

Let $n=3$, by \cite{FGV2}, we need to compute    $\int_{S^*M}{\rm tr}\bigg[\sigma_{-3}\bigg(a^{0} \nabla\left(\left[\tilde{D}, a^{1}\right]\right)|\tilde{D}|^{-3}\bigg)\bigg](x,\xi)$,   $\mathcal{P}_{1}:=\nabla\left(\left[\tilde{D}, a^{1}\right]\right)$Based on the algorithm yielding the principal
symbol of a product of pseudo-differential operators in terms of the principal symbols of the factors, we have

\begin{align}
\sigma_{-3}\bigg(a^{0} \mathcal{P}_{1}|\tilde{D}|^{-3}\bigg)&=a^{0}\left\{\sum_{|\alpha|=0}^\infty\frac{(-i)^{|\alpha|}}{\alpha!}\partial^\alpha_\xi[\sigma(\mathcal{P}_{1})]\partial^\alpha_x[\sigma(|\tilde{D}|^{-3})]\right\}_{-3}\nonumber\\
&=a^{0}\sigma_0(\mathcal{P}_{1})\sigma_{-3}(|\tilde{D}|^{-3})+a^{0}\sigma_1(\mathcal{P}_{1})\sigma_{-4}(|\tilde{D}|^{-3})+(-i)a^{0}\sum_{j=1}^n\partial_{\xi_j}[\sigma_1(\mathcal{P}_{1})]\partial_{x_j}[\sigma_{-3}(|\tilde{D}|^{-3})].
\end{align}

$\mathbf{(ii-1)}$ For $a^{0}\sigma_0(\mathcal{P}_{1})\sigma_{-3}(|\tilde{D}|^{-3})$:\\

According to \eqref{lem3.4.2} in lemma \ref{lemma3.4} and \eqref{lem3.5.2} in lemma \ref{lemma3.5}, we get

\begin{align}
	&\sigma_0(\mathcal{P}_{1})\sigma_{-3}(|\tilde{D}|^{-3})(x_0)\nonumber\\
	&=\sqrt{-1}|\xi|^{-3
	}\sum_{j,\gamma=1}^{n}\partial_{x_j}(X_\gamma)[c(e_j)c(e_\gamma)c(da^1)-c(da^1)c(e_j)c(e_\gamma)]\nonumber\\
	&-2\sqrt{-1}|\xi|^{-3}\sum_{j,\gamma=1}^{n}g(e_j,X)\partial_{x_j}[e_\gamma(a^1)]c(e_\gamma)-|\xi|^{-3}\sum_{j,\gamma=1}^{n}\partial_{x_j}^{2}[e_\gamma(a^1)]c(e_\gamma),
\end{align}

based on the relation of the Clifford action, we can obtain the equality
\begin{align}
	&{\rm tr}\bigg(\sqrt{-1}|\xi|^{-3
	}\sum_{j,\gamma=1}^{n}\partial_{x_j}(X_\gamma)[c(e_j)c(e_\gamma)c(da^1)-c(da^1)c(e_j)c(e_\gamma)]\bigg)\bigg|_{|\xi|=1}=\;0,\\
	&{\rm tr}\bigg(-2\sqrt{-1}|\xi|^{-3
	}\sum_{j,\gamma=1}^{n}g(e_j,X)\partial_{x_j}[e_\gamma(a^1)]c(e_\gamma)\bigg)\bigg|_{|\xi|=1}=\;0,\\
	&{\rm tr}\bigg(-|\xi|^{-3}\sum_{j,\gamma=1}^{n}\partial_{x_j}^{2}[e_\gamma(a^1)]c(e_\gamma)\bigg)\bigg|_{|\xi|=1}=\;0.
\end{align}
Then
 \begin{align}
	\int_{|\xi|=1}{\rm tr}\bigg(a^{0}\sigma_0(\mathcal{P}_{1})\sigma_{-3}(|\tilde{D}|^{-3})(x_0)\bigg)\sigma(\xi)=\;0.
\end{align}

$\mathbf{(ii-2)}$ For $a^{0}\sigma_1(\mathcal{P}_{1})\sigma_{-4}(|\tilde{D}|^{-3})$:\\

According to \eqref{lem3.4.3} in lemma \ref{lemma3.4} and \eqref{lem3.5.1} in lemma \ref{lemma3.5}, we get
\begin{align}
&\sigma_1(\mathcal{P}_{1})\sigma_{-4}(|\tilde{D}|^{-3})(x_0)\nonumber\\
&=-3\sqrt{-1}|\xi|^{-5}\sum_{a,j,\gamma=1}^{n}\partial_{x_j}[e_\gamma(a^1)]\{c(e_a),c(X)\}c(e_\gamma) \xi_a\xi_j,
\end{align}
based on the relation of the Clifford action, we get
\begin{align}
	{\rm tr}\bigg(-3\sqrt{-1}|\xi|^{-5}\sum_{a,j,\gamma=1}^{n}\partial_{x_j}[e_\gamma(a^1)]\{c(e_a),c(X)\}c(e_\gamma) \xi_a\xi_j\bigg)\bigg|_{|\xi|=1}=\;0,
\end{align}
then
 \begin{align}
	\int_{|\xi|=1}{\rm tr}\bigg(a^{0}\sigma_1(\mathcal{P}_{1})\sigma_{-4}(|\tilde{D}|^{-3})(x_0)\bigg)\sigma(\xi)=\;0.
\end{align}

$\mathbf{(ii-3)}$ For $(-i)a^{0}\sum_{j=1}^n\partial_{\xi_j}[\sigma_1(\mathcal{P}_{1})]\partial_{x_j}[\sigma_{-3}(|\tilde{D}|^{-3})]$:\\

According to \eqref{lem3.4.2} in lemma \ref{lemma3.4}, we get
\begin{align}
\partial_{x_j}[\sigma_{-3}(|\tilde{D}|^{-3})]=-\frac{1}{2} |\xi|^{-5}\sum_{a,b,c,d=1}^{n} \left(R_{a c b j} x^{c}+ R_{a j b d} x^{d}\right) \xi_{a} \xi_{b},
\end{align}
then
\begin{align}
	\partial_{x_j}[\sigma_{-3}(|\tilde{D}|^{-3})](x_0)=\;0,
\end{align}
we obtain
 \begin{align}
	\int_{|\xi|=1}{\rm tr}\bigg((-i)a^{0}\sum_{j=1}^n\partial_{\xi_j}[\sigma_1(\mathcal{P}_{1})]\partial_{x_j}[\sigma_{-3}(|\tilde{D}|^{-3})](x_0)\bigg)\sigma(\xi)=\;0.
\end{align}

Summing from {\bf (ii-1)} to {\bf (ii-3)} in turn, we get
 \begin{align}
	\int_{|\xi|=1}{\rm tr}\bigg[\sigma_{-3}\bigg(a^{0} \nabla\left(\left[\tilde{D}, a^{1}\right]\right)|\tilde{D}|^{-3}\bigg)(x_0)\bigg]\sigma(\xi)=\;0.
\end{align}

$\mathbf{Part~~iii)}$
Let $n=3$, by \cite{FGV2}, we need to compute    $\int_{S^*M}{\rm tr}\bigg[\sigma_{-3}\bigg(a^{0} \nabla^{2}\left(\left[\tilde{D}, a^{1}\right]\right)|\tilde{D}|^{-5}\bigg)\bigg](x,\xi)$,   $\mathcal{P}_{2}:=\nabla^{2}\left(\left[\tilde{D}, a^{1}\right]\right)$Based on the algorithm yielding the principal
symbol of a product of pseudo-differential operators in terms of the principal symbols of the factors, we have
 
\begin{align}
	\sigma_{-3}\bigg(a^{0}\mathcal{P}_{2}|\tilde{D}|^{-5}\bigg)=a^{0}\sigma_{2}(\mathcal{P}_{2})\sigma_{-5}(|\tilde{D}|^{-5})
\end{align}

According to \eqref{lem3.4.4} in lemma \ref{lemma3.4}, we get
\begin{align}
	\sigma_{-5}(|\tilde{D}|^{-5})(x_0)=|\xi|^{-5}.
\end{align}

\begin{align}
	\mathcal{P}_{2}:=&\nabla^{2}\left(\left[\tilde{D}, a^{1}\right]\right)\nonumber\\
	=&\nabla\left(\left[\tilde{D}^2,c(da^1)\right]\right)\nonumber\\
	=&\left[\tilde{D}^2,[\tilde{D}^2,c(da^1)]\right]\nonumber\\
	:=&\left[\tilde{D}^2,\mathcal{P}_{1}\right],
\end{align}
hence
\begin{align}
	\sigma_2(\mathcal{P}_{2})=&\sigma_{2}\bigg(\left[\tilde{D}^2,\mathcal{P}_{1}\right]\bigg)\nonumber\\
	=&\sigma_2\bigg(\tilde{D}^2 \circ \mathcal{P}_{1} \bigg) - \sigma_2\bigg( \mathcal{P}_{1} \circ \tilde{D}^2 \bigg)\nonumber\\
	=&\sigma_2(\tilde{D}^2 )\sigma_0( \mathcal{P}_{1} ) + \sigma_1(\tilde{D}^2 )\sigma_1( \mathcal{P}_{1} ) +(-\sqrt{-1})\sum_{j=1}^{n}\partial_{\xi_j}\big(\sigma_2(\tilde{D}^2 )\big)\partial_{x_j}\big(\sigma_1(\mathcal{P}_{1} )\big)\nonumber\\
	&- \sigma_1( \mathcal{P}_{1} )\sigma_1(\tilde{D}^2 )-\sigma_0( \mathcal{P}_{1} )\sigma_2(\tilde{D}^2 ) - (-\sqrt{-1})\sum_{j=1}^{n}\partial_{\xi_j}\big(\sigma_1(\mathcal{P}_{1})\big)\partial_{x_j}\big( \sigma_2(\tilde{D}^2 )\big)\nonumber\\
	=& \sigma_1(\tilde{D}^2 )\sigma_1( \mathcal{P}_{1} )- \sigma_1( \mathcal{P}_{1} )\sigma_1(\tilde{D}^2 ) +(-\sqrt{-1})\sum_{j=1}^{n}\partial_{\xi_j}\big(\sigma_2(\tilde{D}^2 )\big)\partial_{x_j}\big(\sigma_1(\mathcal{P}_{1} )\big).
\end{align}

$\mathbf{(iii-1)}$ For $\sigma_1(\tilde{D}^2 )\sigma_1( \mathcal{P}_{1} )- \sigma_1( \mathcal{P}_{1} )\sigma_1(\tilde{D}^2) $:\\

According to \eqref{lem3.3.1} in lemma \ref{lemma3.5} and \eqref{lem3.5.1} in lemma \ref{lemma3.5}, we get
\begin{align}
 \bigg(\sigma_1(\tilde{D}^2 )\sigma_1( \mathcal{P}_{1} )- \sigma_1( \mathcal{P}_{1} )\sigma_1(\tilde{D}^2 )\bigg)(x_0) =0
\end{align}

$\mathbf{(iii-2)}$ For $(-\sqrt{-1})\sum_{j=1}^{n}\partial_{\xi_j}\big(\sigma_2(\tilde{D}^2 )\big)\partial_{x_j}\big(\sigma_1(\mathcal{P}_{1} )\big)$:\\

According to \eqref{lem3.5.3} in lemma \ref{lemma3.5}
 \begin{align} 
	\sigma_{2} \left(	\nabla^2\left(\left[\tilde{D}, a^{1}\right]\right)\right)(x_0)=
	&\;\frac{1}{2}\sum_{i,j,s,t,l=1}^{n}R_{jlst}(x_0)c(e_s)c(e_t)c(da^1)\xi_j\xi_l\nonumber\\
	&-\frac{1}{2}\sum_{i,j,s,t,l=1}^{n}R_{jlst}(x_0)c(da^1)c(e_s)c(e_t)\xi_j\xi_l\nonumber\\
	&-4\sum_{j,l,\gamma=1}^{n}\partial_{x_j}\partial_{x_l}[e_\gamma(a^1)](x_0)c(e_\gamma)\xi_j\xi_l,
\end{align}
then
\begin{align}
\sigma_{-3}\bigg(a^{0} \nabla^{2}\left(\left[\tilde{D}, a^{1}\right]\right)|\tilde{D}|^{-5}\bigg)(x_0)=
&\;\frac{1}{2}|\xi|^{-5}\sum_{i,j,s,t,l=1}^{n}R_{jlst}(x_0)c(e_s)c(e_t)c(da^1)\xi_j\xi_l\nonumber\\
&-\frac{1}{2}|\xi|^{-5}\sum_{i,j,s,t,l=1}^{n}R_{jlst}(x_0)c(da^1)c(e_s)c(e_t)\xi_j\xi_l\nonumber\\
&-4|\xi|^{-5}\sum_{j,l,\gamma=1}^{n}\partial_{x_j}\partial_{x_l}[e_\gamma(a^1)](x_0)c(e_\gamma)\xi_j\xi_l.
\end{align}
Based on the relation of the Clifford action, we get
\begin{align}
	{\rm tr}\bigg(\frac{1}{2}|\xi|^{-5}\sum_{i,j,s,t,l=1}^{n}R_{jlst}(x_0)\big(c(e_s)c(e_t)c(da^1)-c(da^1)c(e_s)c(e_t)\big)\xi_j\xi_l\bigg)\bigg|_{|\xi|=1}=\;0,
\end{align}
\begin{align}
	{\rm tr}\bigg(-4|\xi|^{-5}\sum_{j,l,\gamma=1}^{n}\partial_{x_j}\partial_{x_l}[e_\gamma(a^1)](x_0)c(e_\gamma)\xi_j\xi_l\bigg)\bigg|_{|\xi|=1}=\;0,
\end{align}

we obtain
\begin{align}
	\int_{|\xi|=1}{\rm tr}\bigg[(-\sqrt{-1})a^{0}|\xi|^{5}\sum_{j=1}^{n}\partial_{\xi_j}\big(\sigma_2(\tilde{D}^2 )\big)\partial_{x_j}\big(\sigma_1(\mathcal{P}_{1} )\big)(x_0)\bigg]\sigma(\xi)=\;0.
\end{align}
Summing from {\bf (iii-1)} to {\bf (iii-2)} in turn, we get
\begin{align}
	\int_{|\xi|=1}{\rm tr}\bigg[\sigma_{-3}\bigg(a^{0} \nabla^{2}\left(\left[\tilde{D}, a^{1}\right]\right)|\tilde{D}|^{-5}\bigg)(x_0)\bigg]\sigma(\xi)=\;0.
\end{align}

Then, summing from {\bf part i)} to {\bf part iii)} in turn, we get

\begin{align}
	{\rm{Trace}} \otimes \phi_{1}(a^{0}, a^{1})=4\sqrt{-1}\pi \int_{M} {\rm{tr}}\bigg( a^{0}\left\langle e_{1}^{*} \wedge e_{2}^{*}\wedge e_{3}^{*} , da^1 \wedge \nabla X^{*} \right\rangle\bigg) d{\rm Vol}_M.\nonumber
\end{align}
Based on Definition \ref{CSdefn}, we get
\begin{align}
	{\rm{Trace}} \otimes \phi_{3}\left(a^{0}, a^{1}, a^{2}, a^{3}\right)=&\frac{1}{12} \int\hspace{-1.05em}- a^{0}\left[\tilde{D}, a^{1}\right]\left[\tilde{D}, a^{2}\right]\left[\tilde{D}, a^{3}\right]|\tilde{D}|^{-3} \nonumber\\
	=&\frac{V_3}{12}\int_{M} {\rm{tr}} \bigg( a^{0}da^1 \wedge da^2 \wedge da^3 \bigg),
\end{align}	
where $a^{j}(0 \leqslant j \leqslant 3) \in M_{N}\big(C^{\infty}(M)\big)$ and $V_3 = (2 \pi^2)^{-1}$,
then
\begin{align}
	6\pi k 	{\rm{Trace}} \otimes \phi_{3}\left(AdA+ \frac{2}{3} A \wedge A \wedge A \right)=\frac{k}{4\pi} c_{0} \int_{M} {\rm{tr}} \bigg(A  \wedge d A+\frac{2}{3} A \wedge A  \wedge A \bigg).
\end{align}	

By the above computations and the definition \ref{CSdefn}, we can get Theorem \ref{thm1}.

\section{The Dirac operator with torsion }

Let $M$ be a smooth compact oriented Riemannian $n$-dimensional manifolds without boundary and $N$ be a vector bundle on $M$.
We say that $P$ is a differential operator of Laplace type, if it has locally the form
\begin{equation}\label{p}
	P=-(g^{ij}\partial_i\partial_j+A^i\partial_i+B),
\end{equation}
where $\partial_{i}$  is a natural local frame on $TM,$ $(g^{ij})_{1\leq i,j\leq n}$ is the inverse matrix associated to the metric
matrix  $(g_{ij})_{1\leq i,j\leq n}$ on $M,$ $A^{i}$ and $B$ are smooth sections of $\textrm{End}(N)$ on $M$ (endomorphism).
If $P$ satisfies the form \eqref{p}, then there is a unique
connection $\nabla$ on $N$ and a unique endomorphism $E$ such that
\begin{equation}
	P=-[g^{ij}(\nabla_{\partial_{i}}\nabla_{\partial_{j}}- \nabla_{\nabla^{L}_{\partial_{i}}\partial_{j}})+E],\nonumber
\end{equation}
where $\nabla^{L}$ is the Levi-Civita connection on $M$. Moreover
(with local frames of $T^{*}M$ and $N$), $\nabla_{\partial_{i}}=\partial_{i}+\omega_{i} $
and $E$ are related to $g^{ij}$, $A^{i}$ and $B$ through
\begin{eqnarray}
	&&\omega_{i}=\frac{1}{2}g_{ij}\big(A^{i}+g^{kl}\Gamma_{ kl}^{j} id\big),\nonumber\\
	&&E=B-g^{ij}\big(\partial_{i}(\omega_{j})+\omega_{i}\omega_{j}-\omega_{k}\Gamma_{ ij}^{k} \big),\nonumber
\end{eqnarray}
where $\Gamma_{ kl}^{j}$ is the  Christoffel coefficient of $\nabla^{L}$.

For smooth vector fields $X, Y, Z$ on $M$, let $\nabla^{T}$ denote the metric connection
\begin{equation}
	\langle \nabla_{X}^{T} Y, Z \rangle=\langle \nabla_{X}^{L} Y, Z \rangle + T(X, Y, Z),\nonumber
\end{equation}
where $T$ is a there form.

Let $M$ be an $n=2m$ dimensional ($n\geq 3$) spin  manifold, we can lift $\nabla^T$ to $\nabla^{S(TM),T}$ on $S(TM)$.
The Dirac operator with torsion $D_T$ is defined as:
\begin{align}
	D_T=&\sum_{j=1}^{n}c(e_{j})\nabla_{e_{j}}^{S(TM),T}\nonumber\\
	=&\sum_{j=1}^{n}c(e_{j})\bigg(e_{j}+\frac{1}{4}\sum_{l, t=1}^{n}\langle \nabla_{e_{j}}^{T}e_{l}, e_{t}\rangle c(e_{l})c(e_{t})\bigg)\nonumber\\
	=&\;D+\frac{1}{4}\sum_{j,l,t=1}^{n}T(e_j, e_l, e_t)c(e_{j})c(e_{l})c(e_{t})\nonumber\\
	=&\;D+\frac{3}{2}\sum_{1\leqslant j< l< t\leqslant n}T(e_j, e_l, e_t)c(e_{j})c(e_{l})c(e_{t}),\nonumber
\end{align}
where $D$ is the Dirac operator induced by the Levi-Civita connection, and $c(e_{j})$ be the Clifford action which satisfies the relation
\begin{align}
	&c(e_{i})c(e_{j})+c(e_{j})c(e_{i})=-2g^{M}(e_{i}, e_{j})=-2\delta_i^j.\nonumber
\end{align}

Let $D_T=D+fc(e_1)c(e_2)c(e_3)$, similarly to lemma \ref{lemma3.3}, we have respective symbols of $D_T^{2}$:
\begin{lem}\label{lemma4.1}
	\begin{align}
		\sigma_1(D_T^2)=&\sqrt{-1} \sum_{\mu=1}^{n} (\Gamma^\mu-2\sigma^\mu)\xi_\mu+\sqrt{-1}fc(e_1)c(e_2)c(e_3)c(\xi)+\sqrt{-1}fc(\xi)c(e_1)c(e_2)c(e_3),\nonumber\\
		\sigma_0(D_T^2)=&- \sum_{\mu,\nu=1}^{n} g^{\mu\nu}(\partial_\mu\sigma_\nu+\sigma_\mu\sigma_\nu-\Gamma^\alpha_{\mu\nu}\sigma_\alpha+fc(e_1)c(e_2)c(e_3)c(e_\nu) \sigma_\nu+ \sum_{\nu=1}^{n} fc(e_\nu)\sigma_\nu c(e_1)c(e_2)c(e_3))\nonumber\\
		&+\sum_{\mu=1}^{n} \frac{\partial f}{\partial_{x_\mu}}c(e_\mu)c(e_1)c(e_2)c(e_3) +f^2+\frac{1}{4}s.
	\end{align}
\end{lem}

\begin{lem}\label{lemma4.2}
	In normal coordinates around a fixed point of the manifold $M$, the symbols representation of the higher inverse of the Laplace operator with torsion read

	\begin{align} 
		\sigma_{-3}(|D_T|^{-1}) & =-\frac{1}{8}|\xi|^{-3} R +\frac{1}{4}|\xi|^{-5}\sum_{a,b=1}^{n}  \operatorname{Ric}_{a b}  \xi_{a} \xi_{b}\nonumber\\
		&-\frac{1}{2}|\xi|^{-3}\sum_{a=1}^{n}\left(\frac{\partial f}{\partial_{x_a}} c(e_a) c(e_1)c(e_2)c(e_3)+f^{2}\right) \nonumber\\
		& +\frac{3}{4}|\xi|^{-5} \sum_{a,b=1}^{n}\frac{\partial f}{\partial_{x_b}}\bigl\{c(e_a), c(e_1)c(e_2)c(e_3)\bigr\}\xi_{a} \xi_{b}\nonumber\\
		& -\frac{3}{8}|\xi|^{-5} f^2\sum_{a,b=1}^{n}\bigl\{c(e_a), c(e_1)c(e_2)c(e_3)\bigr\}\bigl\{c(e_b), c(e_1)c(e_2)c(e_3)\bigr\}\xi_{a} \xi_{b}+o(\mathbf{1})\label{4.2.1}\\
		\sigma_{-3}(|D_T|^{-3})=&|\xi|^{-5}\sum_{a,b,c,d=1}^{n} \left(\delta_{a b}-\frac{1}{2} R_{a c b d} x^{c} x^{d}\right) \xi_{a} \xi_{b}+O\left(\mathbf{x}^{2}\right),\label{lem4.2.2} \\
		\sigma_{-4}(|D_T|^{-3})=&-\sqrt{-1}|\xi|^{-5}\sum_{a,b=1}^{n}  \operatorname{Ric}_{a b} x^{b} \xi_{a}\label{lem4.2.3}\nonumber\\
		&+\frac{ 3 \sqrt{-1}}{8}|\xi|^{-5}\sum_{a,b,j,k=1}^{n} \operatorname{R}_{b a j k}(x_0)c(e_j)c(e_k) x^{b} \xi_{a}\nonumber\\
		&-\frac{3\sqrt{-1}}{2}|\xi|^{-5}\sum_{a=1}^{n}f\bigl\{c(e_a),c(e_1)c(e_2)c(e_3)\bigr\} \xi_a +O\left(\mathbf{x}\right),\\
		\sigma_{-5}(|D_T|^{-5})=&|\xi|^{-7}\sum_{a,b,c,d=1}^{n} \left(\delta_{a b}-\frac{5}{6} R_{a c b d} x^{c} x^{d}\right) \xi_{a} \xi_{b}+O\left(\mathbf{x}^{2}\right),\label{lem4.2.4} 
	\end{align}
	where ${\rm R}_{a c b d}$ and $\operatorname{Ric}_{a b}$ are the components of the Riemann and Ricci tensor, s is the scalar curvature.
\end{lem}

\begin{lem}\label{lemma4.3}

\begin{align} 
	\sigma_{1} \left[D_T^{2}, c(da^1)\right](x_0)=&\;-2\sqrt{-1}\sum_{j,\gamma=1}^{n}\partial_{x_j}[e_\gamma(a^1)]c(e_\gamma)\xi_j\nonumber\\
	&+\sqrt{-1}f\bigg(c(e_1)c(e_2)c(e_3)c(\xi)c(da^1)+c(\xi)c(e_1)c(e_2)c(e_3)c(da^1)\bigg)\nonumber\\
	&-\sqrt{-1}f\bigg(c(da^1)c(e_1)c(e_2)c(e_3)c(\xi)+c(da^1)c(\xi)c(e_1)c(e_2)c(e_3)\bigg),\label{lem4.3.1}\\
	\sigma_{0} \left[D_T^{2}, c(da^1)\right](x_0)	=&\sum_{\mu=1}^{n}\frac{\partial f}{\partial_{x_\mu}}c(e_\mu)c(e_1)c(e_2)c(e_3)c(da^1)-\sum_{\mu=1}^{n}\frac{\partial f}{\partial_{x_\mu}}c(da^1)c(e_\mu)c(e_1)c(e_2)c(e_3) \nonumber\\
	&+\sum_{j,\gamma=1}^{n}(fc(e_1)c(e_2)c(e_3)c(e_j)+fc(e_j)c(e_1)c(e_2)c(e_3))\partial_{x_j}[e_\gamma(a^1)]c(e_\gamma)\nonumber\\
	&-\sum_{j,\gamma=1}^{n}\partial_{x_j}^{2}[e_\gamma(a^1)]c(e_\gamma),\label{lem4.3.2}\\
	\sigma_{2} \left( \nabla^2\left(\left[D_T, a^{1}\right]\right)\right)(x_0)
	=&-f(x_0)\bigg(c(I)c(\xi)c(I)c(\xi)c(da^1)+c(\xi)c(I)c(\xi)c(I)c(da^1)\bigg)\nonumber\\
	&+f(x_0)\bigg(c(da^1)c(I)c(\xi)c(I)c(\xi)+c(da^1)c(\xi)c(I)c(\xi)c(I)\bigg)\nonumber\\
	&-\frac{1}{2}\sum_{j,\mu, t, s=1}^{n} R_{j \mu t s}(x_0)c(e_s)c(e_t)c(da^1)\xi_j \xi_\mu \nonumber\\ &+\frac{1}{2} \sum_{j,\mu, t, s=1}^{n}R_{j \mu t s}(x_0)c(da^1)c(e_s)c(e_t)\xi_j \xi_\mu \nonumber\\
	&+2\sum_{j,l=1}^{n}\frac{\partial f}{\partial_{x_j}}(x_0)\bigg(c(I)c(e_l)c(da^1)+c(e_l)c(I)c(da^1)\bigg)\xi_j \xi_l\nonumber\\
	&+4f\sum_{j,l,\gamma=1}^{n}\partial_{x_j}[e_\gamma(a^1)](x_0)\bigg(c(I)c(e_l)c(r_\gamma)+c(e_l)c(I)c(e_\gamma)\bigg)\xi_j \xi_l\nonumber\\
	&-2\sum_{j,l=1}^{n}\frac{\partial f}{\partial_{x_j}}(x_0)\bigg(c(da^1)c(I)c(e_l)+c(da^1)c(e_l)c(I)\bigg)\xi_j \xi_l\nonumber\\
	&-4f\sum_{j,l,\gamma=1}^{n}\partial_{x_j}[e_\gamma(a^1)](x_0)\bigg(c(r_\gamma)c(I)c(e_l)+c(e_\gamma)c(e_l)c(I)\bigg)\xi_j \xi_l\nonumber\\	
	&+8\sqrt{-1}\sum_{j,l,\gamma=1}^{n}\partial_{x_j}\partial_{x_l}[e_\gamma(a^1)](x_0)c(e_\gamma)\xi_j \xi_l\label{lem4.3.3},
\end{align}
where $c(I):=c(e_1)c(e_2)c(e_3)$.
\end{lem}

\begin{proof}
	
We start with the observation that
\begin{align} 
	\nabla\left(\left[D_T, a^{1}\right]\right)&=\left[D_T^{2}, c(da^1)\right]\nonumber\\
	&= D_T^{2} \circ c(da^1)- c(da^1) \circ D_T^{2}.
\end{align}
Next, the principal symbol of a product of pseudo-differential operators in terms of the principal symbols of the factors, namely:
\begin{align} 
	\sigma\bigg(D_T^{2} \circ c(da^1)\bigg)=&\sum_{|\beta|=0}^\infty\frac{(-\sqrt{-1})^{|\beta|}}{\beta!}\partial^\beta_\xi(\sigma(D_T^{2}))\partial^\beta_x[\sigma(c(da^1))\nonumber\\
	=&\sigma_2(D_T^{2})c(da^1)+\sigma_1(D_T^{2})c(da^1)+(-\sqrt{-1})\sum_{j=1}^{n}\partial_{\xi_j}(\sigma_2(D_T^{2}))\partial{x_j}(c(da^1))\nonumber\\
	&+\sigma_0(D_T^{2})c(da^1)+(-\sqrt{-1})\sum_{j=1}^{n}\partial_{\xi_j}(\sigma_1(D_T^{2}))\partial{x_j}(c(da^1))\nonumber\\
	&-\sum_{j,l=1}^{n}\partial_{\xi_j}\partial_{\xi_l}(\sigma_2(D_T^{2}))\partial{x_j}\partial{x_l}(c(da^1)),\\ 
	\sigma\bigg( c(da^1)  \circ D_T^{2} \bigg)=&\sum_{|\beta|=0}^\infty\frac{(-\sqrt{-1})^{|\beta|}}{\beta!}\partial^\beta_\xi[\sigma(c(da^1))]\partial^\beta_x[\sigma(D_T^{2})]\nonumber\\
	=&c(da^1)\sigma_2(D_T^{2})+c(da^1)\sigma_1(D_T^{2})+c(da^1)\sigma_0(D_T^{2}),
\end{align}
where
\begin{align} 
	\sigma_{1} \left[D_T^{2}, c(da^1)\right]=&\sigma_1(D_T^{2})c(da^1)-c(da^1)\sigma_1(D_T^{2})+(-\sqrt{-1})\sum_{j=1}^{n}\partial_{\xi_j}(\sigma_2(D_T^{2}))\partial{x_j}(c(da^1))\nonumber\\
	=&-2\sqrt{-1}\sigma^\mu c(da^1)\xi_\mu +2\sqrt{-1}c(da^1)\sigma^\mu \xi_\mu-2\sqrt{-1}\sum_{j,\gamma=1}^{n}\partial_{x_j}[e_\gamma(a^1)]c(e_\gamma)\xi_j \nonumber\\
	&+\sqrt{-1}f\bigg(c(e_1)c(e_2)c(e_3)c(\xi)c(da^1)+c(\xi)c(e_1)c(e_2)c(e_3)c(da^1)\bigg)\nonumber\\
	&-\sqrt{-1}f\bigg(c(da^1)c(e_1)c(e_2)c(e_3)c(\xi)+c(da^1)c(\xi)c(e_1)c(e_2)c(e_3)\bigg)\label{4.13} \\
	\sigma_{0} \left[D_T^{2}, c(da^1)\right]=&\sigma_0(D_T^{2})c(da^1)-c(da^1)\sigma_0(D_T^{2})+(-\sqrt{-1})\sum_{j=1}^{n}\partial_{\xi_j}(\sigma_1(D_T^{2}))\partial{x_j}(c(da^1))\nonumber\\
	&-\frac{1}{2}\sum_{j,l=1}^{n}\partial_{\xi_j}\partial_{\xi_l}(\sigma_2(D_T^{2}))\partial{x_j}\partial{x_l}(c(da^1))\nonumber\\
	=&-\sum_{\mu,\nu=1}^{n}g^{\mu\nu}(\partial^{x}_\mu\sigma_\nu+\sigma_\mu\sigma_\nu-\Gamma^\alpha_{\mu\nu}\sigma_\alpha+fc(e_1)c(e_2)c(e_3)c(e_\nu) \sigma_\nu+ f c(e_\nu)\sigma_\nu c(e_1)c(e_2)c(e_3))c(da^1)\nonumber\\
	&+\sum_{\mu=1}^{n}\frac{\partial f}{\partial_{x_\mu}}c(e_\nu)c(e_1)c(e_2)c(e_3)c(da^1)-\sum_{\mu=1}^{n}\frac{\partial f}{\partial_{x_\mu}}c(da^1)c(e_\nu)c(e_1)c(e_2)c(e_3) \nonumber\\
	&+\sum_{\mu,\nu=1}^{n}g^{\mu\nu}c(da^1)(\partial^{x}_\mu\sigma_\nu+\sigma_\mu\sigma_\nu-\Gamma^\alpha_{\mu\nu}\sigma_\alpha+fc(e_1)c(e_2)c(e_3)c(e_\nu) \sigma_\nu+f c(e_\nu)\sigma_\nu c(e_1)c(e_2)c(e_3))\nonumber\\
	&+\sum_{j,\gamma=1}^{n}(\Gamma^j-2\sigma^j+fc(e_1)c(e_2)c(e_3)c(e_j)+fc(e_j)c(e_1)c(e_2)c(e_3))\partial_{x_j}[e_\gamma(a^1)]c(e_\gamma)\nonumber\\
	&-\sum_{j,\gamma=1}^{n}\partial_{x_j}^{2}[e_\gamma(a^1)]c(e_\gamma).\label{4.14}
\end{align}

Based on $\nabla(T)=\left[D^2, T\right]$, the following equation is obtained:
\begin{align} 
	\nabla^2\left(\left[D_T, a^{1}\right]\right)=\left[D_T^{2}, [D_T^{2}, c(da^1)]\right].
\end{align}
Let $\mathcal{P}_{1}:=\nabla\left(\left[D_T, a^{1}\right]\right)$, then
\begin{align} 
	\nabla^2\left(\left[D_T, a^{1}\right]\right)= D_T^{2} \circ \mathcal{P}_{1}- \mathcal{P}_{1} \circ D_T^{2}.
\end{align}
 By calculating the principal symbol of a product of pseudo-differential operators, we obtain
\begin{align} 
	\sigma_{2}\bigg(D_T^{2} \circ \mathcal{P}_{1}\bigg)=&\sigma_2(D_T^{2})\sigma_0(\mathcal{P}_{1})+\sigma_1(D_T^{2})\sigma_1(\mathcal{P}_{1})+(-\sqrt{-1})\sum_{j=1}^{n}\partial_{\xi_j}(\sigma_2(D_T^{2}))\partial{x_j}(\sigma_1(\mathcal{P}_{1})),\\ 
	\sigma_{2}\bigg(\mathcal{P}_{1} \circ D_T^{2}\bigg)=&\sigma_1(\mathcal{P}_{1})\sigma_1(D_T^{2})+\sigma_0(\mathcal{P}_{1})\sigma_2(D_T^{2})+(-\sqrt{-1})\sum_{j=1}^{n}\partial_{\xi_j}(\sigma_1(\mathcal{P}_{1})\partial{x_j}(\sigma_2(D_T^{2})).
\end{align}

According to \eqref{4.13} and \eqref{4.14}, we obtain
\begin{align} 
	\sigma_{2} \left(	\nabla^2\left(\left[D_T, a^{1}\right]\right)\right)=&\sigma_{2}\bigg(D_T^{2} \circ \mathcal{P}_{1}\bigg)-\sigma_{2}\bigg(\mathcal{P}_{1} \circ D_T^{2}\bigg)\nonumber\\
	=&\sigma_1(D_T^{2})\sigma_1(\mathcal{P}_{1})-\sigma_1(\mathcal{P}_{1})\sigma_1(D_T^{2})+(-\sqrt{-1})\sum_{j=1}^{n}\partial_{\xi_j}(\sigma_2(D_T^{2}))\partial{x_j}(\sigma_1(\mathcal{P}_{1})),
\end{align}
Using \eqref{2.31} to \eqref{2.34}, obtain
\begin{align} 
	&\bigg(\sigma_1(D_T^{2})\sigma_1(\mathcal{P}_{1})-\sigma_1(\mathcal{P}_{1})\sigma_1(D_T^{2})\bigg)(x_0)\\
	=&\;\;2f(x_0)\sum_{j,l,\gamma=1}^{n}\partial_{x_j}[e_\gamma(a^1)]\bigg(c(I)c(e_l)c(e_\gamma)+c(e_l)c(I)c(e_\gamma)\bigg)\xi_j\xi_l\nonumber\\
	&-2f(x_0)\sum_{j,l,\gamma=1}^{n}\partial_{x_j}[e_\gamma(a^1)]\bigg(c(e_\gamma)c(I)c(e_l)+c(e_\gamma)c(e_l)c(I)\bigg)\xi_j\xi_l\nonumber\\
	&-f(x_0)\bigg(c(I)c(\xi)c(I)c(\xi)c(da^1)+c(\xi)c(I)c(\xi)c(I)c(da^1)\bigg)\nonumber\\
	&+f(x_0)\bigg(c(da^1)c(I)c(\xi)c(I)c(\xi)+c(da^1)c(\xi)c(I)c(\xi)c(I)\bigg),
\end{align}
and
\begin{align} 
\sum_{j=1}^{n}\partial_{\xi_j}(\sigma_2(D_T^{2}))(x_0)=&2\sum_{j=1}^{n}\xi_j,\\
    \sum_{j=1}^{n}\partial{x_j}(\sigma_1(\mathcal{P}_{1}))(x_0)=&-\frac{\sqrt{-1}}{4}\sum_{j,\mu, t, s=1}^{n} R_{j \mu t s}(x_0)c(e_s)c(e_t)c(da^1)\xi_\mu \nonumber\\
	& +\frac{\sqrt{-1}}{4} \sum_{j,\mu, t, s=1}^{n}R_{j \mu t s}(x_0)c(da^1)c(e_s)c(e_t)\xi_\mu \nonumber\\
	&+\sqrt{-1}\sum_{j=1}^{n}\frac{\partial f}{\partial_{x_j}}(x_0)\bigg(c(I)c(\xi)c(da^1)+c(\xi)c(I)c(da^1)\bigg)\nonumber\\
	&+\sqrt{-1}f\sum_{j,\gamma=1}^{n}\partial_{x_j}[e_\gamma(a^1)](x_0)\bigg(c(I)c(\xi)c(r_\gamma)+c(\xi)c(I)c(e_\gamma)\bigg)\nonumber\\
	&-\sqrt{-1}\sum_{j=1}^{n}\frac{\partial f}{\partial_{x_j}}(x_0)\bigg(c(da^1)c(I)c(\xi)+c(da^1)c(\xi)c(I)\bigg)\nonumber\\
	&-\sqrt{-1}f\sum_{j,\gamma=1}^{n}\partial_{x_j}[e_\gamma(a^1)](x_0)\bigg(c(r_\gamma)c(I)c(\xi)+c(e_\gamma)c(\xi)c(I)\bigg)\nonumber\\	
	&-4\sum_{j,l,\gamma=1}^{n}\partial_{x_j}\partial_{x_l}[e_\gamma(a^1)](x_0)c(e_\gamma)\xi_l
\end{align}

Then, we obtain Lemma \ref{lemma4.3}.

\end{proof}

To prove Theorem \ref{thm2}, based on Definition \ref{CSdefn}, we first compute $\phi_{1}(A)$, for $D_T=D+fc(e_1)c(e_2)c(e_3)$, $A=a^{0}d a^{1} $, and $a^{0}, a^{1} \in M_{N}\big(C^{\infty}(M)\big),$ where
\begin{align}
	\varphi_{1}\left(a^{0}, a^{1}\right)=&\int\hspace{-1.05em}- a^{0}\left[D_T, a^{1}\right]|D_T|^{-1}  -\frac{1}{4} \int\hspace{-1.05em}- a^{0} \nabla\left(\left[D_T, a^{1}\right]\right)|D_T|^{-3}  +\frac{1}{8} \int\hspace{-1.05em}- a^{0} \nabla^{2}\left(\left[D_T, a^{1}\right]\right)|D_T|^{-5}
\end{align}

$\mathbf{Part~~I)}$
Let $n=3$, by \cite{FGV2}, we need to compute  $\int_{S^*M}{\rm tr}\bigg[\sigma_{-3}\bigg(a^{0}\left[D_T, a^{1}\right]|D_T|^{-1}\bigg)\bigg](x,\xi).$  Based on the algorithm yielding the principal
symbol of a product of pseudo-differential operators in terms of the principal symbols of the factors, we have
\begin{align}
	\sigma_{-3}\bigg(a^{0}\left[D_T, a^{1}\right]|D_T|^{-1}\bigg)(x_0)
	=&a^{0}c(da^1)\sigma_{-3}(|D_T|^{-1})(x_0)\nonumber\\
	=&-\frac{1}{8}a^{0}|\xi|^{-3} R c(da^1) +\frac{1}{4}a^{0}|\xi|^{-5}\sum_{a,b=1}^{n}  \operatorname{Ric}_{a b} c(da^1) \xi_{a} \xi_{b}\nonumber\\
	&-\frac{1}{2}a^{0}|\xi|^{-3}\sum_{a=1}^{n}f^{2}c(da^1)-\frac{1}{2}a^{0}|\xi|^{-3}\sum_{a=1}^{n}\frac{\partial f}{\partial_{x_a}}c(da^1) c(e_a) c(e_1)c(e_2)c(e_3) \nonumber\\
	& +\frac{3}{4}a^{0}|\xi|^{-5} \sum_{a,b=1}^{n}\frac{\partial f}{\partial_{x_b}}c(da^1)\bigl\{c(e_a), c(e_1)c(e_2)c(e_3)\bigr\}\xi_{a} \xi_{b}\nonumber\\
	& -\frac{3}{8}a^{0}|\xi|^{-5} f^2\sum_{a,b=1}^{n}c(da^1)\bigl\{c(e_a), c(e_1)c(e_2)c(e_3)\bigr\}\bigl\{c(e_b), c(e_1)c(e_2)c(e_3)\bigr\}\xi_{a} \xi_{b}.
\end{align}

 Below, we compute each term of  $\int_{|\xi|=1} \operatorname{tr} \bigg[	\sigma_{-3}\bigg(a^{0}\left[\tilde{D}, a^{1}\right]|\tilde{D}|^{-1}\bigg)\bigg](x, \xi) \sigma(\xi)$ in turn.  
When $n=3$, $\rm{tr}[id]=2 $ and $\rm{tr}\big( c(e_1) c(e_2) c(e_3)\big)=-2$. Then, based on the relation of the Clifford action $\operatorname{tr}\big(c(\mathcal{X})c(\mathcal{Y})\big)=-g(\mathcal{X},\mathcal{Y}) \rm{tr}[id]$, we get the following equations.

$\mathbf{(I-1)}$
\begin{align}
	&{\rm tr}\bigg(-\frac{1}{8}a^{0}|\xi|^{-3} R c(da^1)\bigg)\bigg|_{|\xi|=1}=\;0,\\
	&{\rm tr}\bigg(\frac{1}{4}a^{0}|\xi|^{-5} \sum_{a,b=1}^{n}\operatorname{Ric}_{a b} c(da^1) \xi_{a} \xi_{b}\bigg)\bigg|_{|\xi|=1}=\;0,\\
	&{\rm tr}\bigg(-\frac{1}{2}a^{0}|\xi|^{-3}f^{2}c(da^1)\bigg)\bigg|_{|\xi|=1}=\;0,
\end{align}

$\mathbf{(I-2)}$
Based on the relation of the Clifford action, we get
\begin{align}
	&{\rm tr}\bigg(-\frac{1}{2}a^{0}|\xi|^{-3}\sum_{a=1}^{n}\frac{\partial f}{\partial_{x_a}}c(da^1) c(e_a) c(e_1)c(e_2)c(e_3)\bigg)\bigg|_{|\xi|=1}\nonumber\\
	=&{\rm tr}\bigg(-\frac{1}{2}a^{0}\sum_{a,\gamma=1}^{n}\frac{\partial f}{\partial_{x_a}}e_{\gamma}(a^1)c(e_\gamma) c(e_a) c(e_1)c(e_2)c(e_3)\bigg)\bigg|_{|\xi|=1}\nonumber\\
	=&\left\{\begin{array}{ll}0, 
		& \text { if } a\neq\gamma; \\
		-a^{0}\sum_{a=1}^{n}\frac{\partial f}{\partial_{x_a}}e_{a}(a^1), 
		& \text { if } a= \gamma,
	\end{array}\right.
\end{align}
then
\begin{align}\label{4.30}
	&\int_{|\xi|=1}{\rm tr}\bigg(-\frac{1}{2}a^{0}|\xi|^{-3}\sum_{a=1}^{n}\frac{\partial f}{\partial_{x_a}}c(da^1) c(e_a) c(e_1)c(e_2)c(e_3)\bigg)\sigma(\xi)\nonumber\\
	&= -a^{0}\sum_{a=1}^{n}\frac{\partial f}{\partial_{x_a}}e_{a}(a^1) {\rm Vol}(S^{2}).
\end{align}

$\mathbf{(I-3)}$
Based on the relation of the Clifford action, we get
\begin{align}
	&{\rm tr}\bigg(\frac{3}{4}a^{0}|\xi|^{-5} \sum_{a,b=1}^{n}\frac{\partial f}{\partial_{x_b}}c(da^1)\bigl\{c(e_a), c(e_1)c(e_2)c(e_3)\bigr\}\xi_{a} \xi_{b}\bigg)\bigg|_{|\xi|=1}\nonumber\\
	=&{\rm tr}\bigg(\frac{3}{4}a^{0}\sum_{a,b,\gamma=1}^{n}\frac{\partial f}{\partial_{x_b}}e_{\gamma}(a^1)c(e_\gamma)\bigl\{c(e_a), c(e_1)c(e_2)c(e_3)\bigr\}\xi_{a} \xi_{b}\bigg)\nonumber\\
	=&\frac{3}{4}a^{0}\sum_{a,b,\gamma=1}^{n}\frac{\partial f}{\partial_{x_b}}e_{\gamma}(a^1)\xi_{a} \xi_{b}{\rm tr}\bigg(c(e_\gamma)c(e_a) c(e_1)c(e_2)c(e_3)+c(e_\gamma) c(e_1)c(e_2)c(e_3)c(e_a)\bigg)\nonumber\\
	=&\left\{\begin{array}{ll}0, 
		& \text { if } a\neq\gamma; \\
		3a^{0}\sum_{a=1}^{n}\frac{\partial f}{\partial_{x_b}}e_{a}(a^1)\xi_{a} \xi_{b}, 
		& \text { if } a= \gamma,
	\end{array}\right.
\end{align}
then
\begin{align}\label{4.32}
	&\int_{|\xi|=1}{\rm tr}\bigg(\frac{3}{4}a^{0}|\xi|^{-5} \sum_{a,b=1}^{n}\frac{\partial f}{\partial_{x_b}}c(da^1)\bigl\{c(e_a), c(e_1)c(e_2)c(e_3)\bigr\}\xi_{a} \xi_{b}\bigg)\sigma(\xi)\nonumber\\
	&=a^{0}\sum_{a=1}^{n}\frac{\partial f}{\partial_{x_a}}e_{a}(a^1) {\rm Vol}(S^{2})
\end{align}

$\mathbf{(I-4)}$
Based on the relation of the Clifford action, we get
\begin{align}
	&{\rm tr}\bigg(\sum_{a=1}^{n}c(da^1)\bigl\{c(e_a), c(e_1)c(e_2)c(e_3)\bigr\}\bigl\{c(e_a), c(e_1)c(e_2)c(e_3)\bigr\}\bigg)\nonumber\\
	=&\sum_{a,\gamma=1}^{n}e_\gamma(a^1){\rm tr}\bigg(c(e_\gamma)c(e_a) c(e_1)c(e_2)c(e_3)c(e_a) c(e_1)c(e_2)c(e_3)\nonumber\\
	&\;\;\;\;\;\;\;\;\;\;\;\;\;\;\;\;\;\;\;\;\;+c(e_\gamma)c(e_a) c(e_1)c(e_2)c(e_3) c(e_1)c(e_2)c(e_3)c(e_a)\nonumber\\
	&\;\;\;\;\;\;\;\;\;\;\;\;\;\;\;\;\;\;\;\;\;+c(e_\gamma) c(e_1)c(e_2)c(e_3)c(e_a)c(e_a) c(e_1)c(e_2)c(e_3)\nonumber\\
	&\;\;\;\;\;\;\;\;\;\;\;\;\;\;\;\;\;\;\;\;\;+c(e_\gamma) c(e_1)c(e_2)c(e_3)c(e_a) c(e_1)c(e_2)c(e_3)c(e_a)\bigg)\nonumber\\
	&=\;0,
\end{align}
then
\begin{align}\label{4.34}
	&\int_{|\xi|=1}{\rm tr}\bigg(-\frac{3}{8}a^{0}|\xi|^{-5} f^2\sum_{a,b=1}^{n}c(da^1)\bigl\{c(e_a), c(e_1)c(e_2)c(e_3)\bigr\}\bigl\{c(e_b), c(e_1)c(e_2)c(e_3)\bigr\}\xi_{a} \xi_{b}\bigg)\sigma(\xi)\nonumber\\
	&=\;0.
\end{align}
Summing from {\bf (I-1)} to {\bf (I-4)} in turn, we get
\begin{align}
	&\int_{|\xi|=1}{\rm tr}\bigg(	\sigma_{-3}\big(a^{0}\left[D_T, a^{1}\right]|D_T|^{-1}\big)(x_0)\bigg)\sigma(\xi)\nonumber\\
	&=\;0.
\end{align}

$\mathbf{Part~~II)}$

Let $n=3$, by \cite{FGV2}, we need to compute    $\int_{S^*M}{\rm tr}\bigg[\sigma_{-3}\bigg(a^{0} \nabla\left(\left[D_T, a^{1}\right]\right)|D_T|^{-3}\bigg)\bigg](x,\xi)$,   $\mathcal{P}_{1}:=\nabla\left(\left[D_T, a^{1}\right]\right)$Based on the algorithm yielding the principal
symbol of a product of pseudo-differential operators in terms of the principal symbols of the factors, we have

\begin{align}
	\sigma_{-3}\bigg(a^{0} \mathcal{P}_{1}|D_T|^{-3}\bigg)&=a^{0}\left\{\sum_{|\alpha|=0}^\infty\frac{(-i)^{|\alpha|}}{\alpha!}\partial^\alpha_\xi[\sigma(\mathcal{P}_{1})]\partial^\alpha_x[\sigma(|D_T|^{-3})]\right\}_{-3}\nonumber\\
	&=a^{0}\sigma_0(\mathcal{P}_{1})\sigma_{-3}(|D_T|^{-3})+a^{0}\sigma_1(\mathcal{P}_{1})\sigma_{-4}(|D_T|^{-3})+(-i)a^{0}\sum_{j=1}^n\partial_{\xi_j}[\sigma_1(\mathcal{P}_{1})]\partial_{x_j}[\sigma_{-3}(|D_T|^{-3})].
\end{align}

$\mathbf{(II-1)}$ For $a^{0}\sigma_0(\mathcal{P}_{1})\sigma_{-3}(|D_T|^{-3})$:\\
According to \eqref{lem4.2.2} in lemma \ref{lemma4.2} and \eqref{lem4.3.2} in lemma \ref{lemma4.3}, we get
\begin{align}
	&\sigma_0(\mathcal{P}_{1})\sigma_{-3}(|D_T|^{-3})(x_0)\nonumber\\
	&=|\xi|^{-3}\sum_{\mu=1}^{n}\frac{\partial f}{\partial_{x_\mu}}\bigg(c(e_\mu)c(e_1)c(e_2)c(e_3)c(da^1)-c(da^1)c(e_\mu)c(e_1)c(e_2)c(e_3)\bigg) \nonumber\\
	&+|\xi|^{-3}f\sum_{j,\gamma=1}^{n}\partial_{x_j}[e_\gamma(a^1)]\bigg(c(e_1)c(e_2)c(e_3)c(e_j)c(e_\gamma)+c(e_j)c(e_1)c(e_2)c(e_3)c(e_\gamma)\bigg)\nonumber\\
	&-|\xi|^{-3}\sum_{j,\gamma=1}^{n}\partial_{x_j}^{2}[e_\gamma(a^1)]c(e_\gamma),
\end{align}
based on the relation of the Clifford action, we can obtain the equality
\begin{align}
	{\rm tr}\bigg(|\xi|^{-3}\sum_{\mu=1}^{n}\frac{\partial f}{\partial_{x_\mu}}\big(c(e_\mu)c(e_1)c(e_2)c(e_3)c(da^1)-c(da^1)c(e_\mu)c(e_1)c(e_2)c(e_3)\big)\bigg)\bigg|_{|\xi|=1}=\;0,
\end{align}

\begin{align}
	&{\rm tr}\bigg(|\xi|^{-3}f\sum_{j,\gamma=1}^{n}\partial_{x_j}[e_\gamma(a^1)]\big(c(e_1)c(e_2)c(e_3)c(e_j)c(e_\gamma)+c(e_j)c(e_1)c(e_2)c(e_3)c(e_\gamma)\big)\bigg)\bigg|_{|\xi|=1}\\
	=&f\sum_{j,\gamma=1}^{n}\partial_{x_j}[e_\gamma(a^1)]{\rm tr}\bigg(c(e_1)c(e_2)c(e_3)c(e_j)c(e_\gamma)+c(e_j)c(e_1)c(e_2)c(e_3)c(e_\gamma)\bigg)\nonumber\\
	=&\left\{\begin{array}{ll}0, 
		& \text { if } j\neq\gamma; \\
		4f\sum_{j=1}^{n}\partial_{x_j}[e_j(a^1)], 
		& \text { if } j= \gamma,
	\end{array}\right.
\end{align}
and
\begin{align}
{\rm tr}\bigg(-|\xi|^{-3}\sum_{j,\gamma=1}^{n}\partial_{x_j}^{2}[e_\gamma(a^1)]c(e_\gamma)\bigg)\bigg|_{|\xi|=1}=\;0,
\end{align}
then
\begin{align}
	\int_{|\xi|=1}{\rm tr}\bigg(a^{0}\sigma_0(\mathcal{P}_{1})\sigma_{-3}(|D_T|^{-3})\bigg)\sigma(\xi)=4a^{0}f\sum_{j=1}^{n}\partial_{x_j}[e_j(a^1)]{\rm Vol}(S^{2}).
\end{align}

$\mathbf{(II-2)}$ For $a^{0}\sigma_1(\mathcal{P}_{1})\sigma_{-4}(|D_T|^{-3})$:\\
According to \eqref{lem4.2.3} in lemma \ref{lemma4.2} and \eqref{lem4.3.1} in lemma \ref{lemma4.3}, we get
\begin{align}
	&\sigma_1(\mathcal{P}_{1})\sigma_{-4}(|D_T|^{-3})(x_0)\nonumber\\
	&=-3|\xi|^{-5}f\sum_{a,j,\gamma=1}^{n}\partial_{x_j}[e_\gamma(a^1)]\bigl\{c(e_a),c(e_1)c(e_2)c(e_3)\bigr\}c(e_\gamma)\xi_a\xi_j\nonumber\\
	&+\frac{3}{2}|\xi|^{-5}f^2\sum_{a,j=1}^{n}\bigl\{c(e_a),c(e_1)c(e_2)c(e_3)\bigr\}c(e_1)c(e_2)c(e_3)c(e_j)c(da^1)\xi_a\xi_j\nonumber\\		&+\frac{3}{2}|\xi|^{-5}f^2\sum_{a,j=1}^{n}\bigl\{c(e_a),c(e_1)c(e_2)c(e_3)\bigr\}c(e_j)c(e_1)c(e_2)c(e_3)c(da^1)\xi_a\xi_j\nonumber\\
	&-\frac{3}{2}|\xi|^{-5}f^2\sum_{a,j=1}^{n}\bigl\{c(e_a),c(e_1)c(e_2)c(e_3)\bigr\}c(da^1)c(e_1)c(e_2)c(e_3)c(e_j)\xi_a\xi_j\nonumber\\
	&-\frac{3}{2}|\xi|^{-5}f^2\sum_{a,j=1}^{n}\bigl\{c(e_a),c(e_1)c(e_2)c(e_3)\bigr\}c(da^1)c(e_j)c(e_1)c(e_2)c(e_3)\xi_a\xi_j,
\end{align}
based on the relation of the Clifford action, we can obtain the equality
\begin{align}
	&{\rm tr}\bigg(-3|\xi|^{-5}f\sum_{a,j,\gamma=1}^{n}\partial_{x_j}[e_\gamma(a^1)]\bigl\{c(e_a),c(e_1)c(e_2)c(e_3)\bigr\}c(e_\gamma)\xi_a\xi_j\bigg)\bigg|_{|\xi|=1}\nonumber\\
	=&-3f\sum_{a,j,\gamma=1}^{n}\partial_{x_j}[e_\gamma(a^1)]\xi_a\xi_j{\rm tr}\bigg(c(e_a)c(e_1)c(e_2)c(e_3)c(e_\gamma)+c(e_1)c(e_2)c(e_3)c(e_a)c(e_\gamma)\bigg)\nonumber\\
	=&\left\{\begin{array}{ll}0, 
		& \text { if } a\neq\gamma; \\
		-12f\sum_{a,j=1}^{n}\partial_{x_j}[e_a(a^1)]\xi_a\xi_j, 
		& \text { if } a= \gamma,
	\end{array}\right.
\end{align}
and
\begin{align}
	&{\rm tr}\bigg(\sum_{a=1}^{n}\bigl\{c(e_a),c(e_1)c(e_2)c(e_3)\bigr\}c(e_1)c(e_2)c(e_3)c(e_a)c(da^1)\bigg)\bigg|_{|\xi|=1}=0,\\
	&{\rm tr}\bigg(\sum_{a=1}^{n}\bigl\{c(e_a),c(e_1)c(e_2)c(e_3)\bigr\}c(e_a)c(e_1)c(e_2)c(e_3)c(da^1)\bigg)\bigg|_{|\xi|=1}=0,\\
	&{\rm tr}\bigg(\sum_{a=1}^{n}\bigl\{c(e_a),c(e_1)c(e_2)c(e_3)\bigr\}c(da^1)c(e_1)c(e_2)c(e_3)c(e_a)\bigg)\bigg|_{|\xi|=1}=0,\\
	&{\rm tr}\bigg(\sum_{a=1}^{n}\bigl\{c(e_a),c(e_1)c(e_2)c(e_3)\bigr\}c(da^1)c(e_a)c(e_1)c(e_2)c(e_3)\bigg)\bigg|_{|\xi|=1}=0,
\end{align}
then
\begin{align}
	\int_{|\xi|=1}{\rm tr}\bigg(a^{0}\sigma_1(\mathcal{P}_{1})\sigma_{-4}(|\tilde{D}|^{-3})\bigg)\sigma(\xi)=-4a^{0}f\sum_{a=1}^{n}\partial_{x_a}[e_a(a^1)]{\rm Vol}(S^{2}).
\end{align}

$\mathbf{(II-3)}$ For $(-i)a^{0}\sum_{j=1}^n\partial_{\xi_j}[\sigma_1(\mathcal{P}_{1})]\partial_{x_j}[\sigma_{-3}(|D_T|^{-3})]$:\\
According to \eqref{lem4.2.2} in lemma \ref{lemma4.2} and \eqref{lem4.3.1} in lemma \ref{lemma4.3}, we get
\begin{align}
	\partial_{x_j}[\sigma_{-3}(|D_T|^{-3})]=-\frac{1}{2} |\xi|^{-5}\sum_{a,b,c,d=1}^{n} \left(R_{a c b j} x^{c}+ R_{a j b d} x^{d}\right) \xi_{a} \xi_{b},
\end{align}

then
\begin{align}
	\partial_{x_j}[\sigma_{-3}(|D_T|^{-3})](x_0)=\;0.
\end{align}

\begin{align}
	\int_{|\xi|=1}{\rm tr}\bigg((-i)a^{0}\sum_{j=1}^n\partial_{\xi_j}[\sigma_1(\mathcal{P}_{1})]\partial_{x_j}[\sigma_{-3}(|D_T|^{-3})]\bigg)\sigma(\xi)=\;0.
\end{align}

Summing from {\bf (II-1)} to {\bf (II-3)} in turn, we get
\begin{align}
	\int_{|\xi|=1}{\rm tr}\bigg[\sigma_{-3}\bigg(a^{0} \nabla\left(\left[D_T, a^{1}\right]\right)|D_T|^{-3}\bigg)(x_0)\bigg]\sigma(\xi)=\;0.
\end{align}

$\mathbf{Part~~III)}$
Let $n=3$, by \cite{FGV2}, we need to compute    $\int_{S^*M}{\rm tr}\bigg[\sigma_{-3}\bigg(a^{0} \nabla^{2}\left(\left[D_T, a^{1}\right]\right)|D_T|^{-5}\bigg)\bigg](x,\xi)$,   $\mathcal{P}_{2}:=\nabla^{2}\left(\left[D_T, a^{1}\right]\right)$Based on the algorithm yielding the principal
symbol of a product of pseudo-differential operators in terms of the principal symbols of the factors, we have

\begin{align}
	\sigma_{-3}\bigg(a^{0}\mathcal{P}_{2}|D_T|^{-5}\bigg)=a^{0}\sigma_{2}(\mathcal{P}_{2})\sigma_{-5}(|D_T|^{-5})
\end{align}

According to \eqref{lem4.2.4} in lemma \ref{lemma4.2}, we get
\begin{align}
	\sigma_{-5}(|D_T|^{-5})(x_0)=|\xi|^{-5}.
\end{align}

\begin{align}
	\mathcal{P}_{2}:=&\nabla^{2}\left(\left[D_T, a^{1}\right]\right)\nonumber\\
	=&\nabla\left(\left[D_T^2,c(da^1)\right]\right)\nonumber\\
	=&\left[\tilde{D}^2,[D_T^2,c(da^1)]\right]\nonumber\\
	:=&\left[D_T^2,\mathcal{P}_{1}\right],
\end{align}
hence
\begin{align}
	\sigma_2(\mathcal{P}_{2})=\sigma_1(D_T^2 )\sigma_1( \mathcal{P}_{1} )- \sigma_1( \mathcal{P}_{1} )\sigma_1(D_T^2 ) +(-\sqrt{-1})\sum_{j=1}^{n}\partial_{\xi_j}\big(\sigma_2(D_T^2 )\big)\partial_{x_j}\big(\sigma_1(\mathcal{P}_{1} )\big).
\end{align}
According to \eqref{lem4.2.4} in lemma \ref{lemma4.2} and \eqref{lem4.3.3} in lemma \ref{lemma4.3}, we get
\begin{align} 
	\sigma_{-3}\bigg(a^{0}\mathcal{P}_{2}|D_T|^{-5}\bigg)=&a^{0}\sigma_{2} \left( \nabla^2\left(\left[D_T, a^{1}\right]\right)\right)	\sigma_{-5}(|D_T|^{-5})(x_0)\nonumber\\
	=&-a^{0}|\xi|^{-5}f(x_0)\bigg(c(I)c(\xi)c(I)c(\xi)c(da^1)+c(\xi)c(I)c(\xi)c(I)c(da^1)\bigg)\nonumber\\
	&+a^{0}|\xi|^{-5}f(x_0)\bigg(c(da^1)c(I)c(\xi)c(I)c(\xi)+c(da^1)c(\xi)c(I)c(\xi)c(I)\bigg)\nonumber\\
	&-\frac{1}{2}a^{0}|\xi|^{-5}\sum_{j,\mu, t, s=1}^{n} R_{j \mu t s}(x_0)c(e_s)c(e_t)c(da^1)\xi_j \xi_\mu \nonumber\\ &+\frac{1}{2}a^{0}|\xi|^{-5} \sum_{j,\mu, t, s=1}^{n}R_{j \mu t s}(x_0)c(da^1)c(e_s)c(e_t)\xi_j \xi_\mu \nonumber\\
	&+2a^{0}|\xi|^{-5}\sum_{j,l=1}^{n}\frac{\partial f}{\partial_{x_j}}(x_0)\bigg(c(I)c(e_l)c(da^1)+c(e_l)c(I)c(da^1)\bigg)\xi_j \xi_l\nonumber\\
	&+4a^{0}|\xi|^{-5}f\sum_{j,l,\gamma=1}^{n}\partial_{x_j}[e_\gamma(a^1)](x_0)\bigg(c(I)c(e_l)c(r_\gamma)+c(e_l)c(I)c(e_\gamma)\bigg)\xi_j \xi_l\nonumber\\
	&-2a^{0}|\xi|^{-5}\sum_{j,l=1}^{n}\frac{\partial f}{\partial_{x_j}}(x_0)\bigg(c(da^1)c(I)c(e_l)+c(da^1)c(e_l)c(I)\bigg)\xi_j \xi_l\nonumber\\
	&-4a^{0}|\xi|^{-5}f\sum_{j,l,\gamma=1}^{n}\partial_{x_j}[e_\gamma(a^1)](x_0)\bigg(c(r_\gamma)c(I)c(e_l)+c(e_\gamma)c(e_l)c(I)\bigg)\xi_j \xi_l\nonumber\\	
	&+8\sqrt{-1}a^{0}|\xi|^{-5}\sum_{j,l,\gamma=1}^{n}\partial_{x_j}\partial_{x_l}[e_\gamma(a^1)](x_0)c(e_\gamma)\xi_j \xi_l,
\end{align}
Based on the relation of the Clifford action and ${\rm tr}\mathcal{XY}={\rm tr}\mathcal{YX}$, we get
\begin{align}
	\int_{|\xi|=1}{\rm tr}\bigg(\sigma_{-3}\big(a^{0} \nabla^{2}\left(\left[D_T, a^{1}\right]\right)|D_T|^{-5}\big)\bigg)\sigma(\xi)=\;0.
\end{align}

Summing from {\bf part I)} to {\bf part III)} in turn, we get

\begin{align}
	\rm{Trace} \otimes\varphi_{1}\left(a^{0}, a^{1}\right)=&\int\hspace{-1.05em}- a^{0}\left[D_T, a^{1}\right]|D_T|^{-1}  -\frac{1}{4} \int\hspace{-1.05em}- a^{0} \nabla\left(\left[D_T, a^{1}\right]\right)|D_T|^{-3}  +\frac{1}{8} \int\hspace{-1.05em}- a^{0} \nabla^{2}\left(\left[D_T, a^{1}\right]\right)|D_T|^{-5}\nonumber\\
	=&\;0.
\end{align}
By the above computations and the definition \ref{CSdefn}, we can get Theorem \ref{thm2}.

\section*{ Acknowledgements}
This work is sponsored by Natural Science Foundation of Xinjiang Uygur Autonomous Region 2024D01C341 and supported by the National Natural Science Foundation of China 11771070, 12061078 .
 The authors thank the referee for his (or her) careful reading and helpful comments.

\section*{Declarations}

\begin{itemize}
	\item  {\bf Funding} The Natural Science Foundation of Xinjiang Uygur Autonomous Region 2024D01C341 and the National Natural Science Foundation of China 11771070, 12061078.
	\item {\bf Data availability} The authors declare that The data supporting the findings of this study are available within the paper.
	\item {\bf Conflict of interest} The authors declare no Confict of interest.
	\item {\bf Ethics approval and consent to participate} Ethics approval and consent to participate.
	\item {\bf Consent for publication} All the authors agreed to publish this research.
	\item {\bf Author contribution} All authors contributed to the study conception and design. Material preparation, data collection and analysis were performed by JH and YW. The frst draft of the manuscript was written by JH and all authors commented on previous versions of the manuscript. All authors read and approved the fnal manuscript.
\end{itemize}

\end{document}